\documentclass[a4paper,12pt,reqno]{amsart}
\usepackage{amsfonts, amsthm, amsmath, amssymb,bbold,pdfpages, pgfplots}
\pgfplotsset{width=11cm, height=8cm, compat=1.5}
\renewcommand{\b}[1]{\mathbf{#1}}
\newcommand{\cl}[1]{\mathcal{#1}}
\newcommand{\card}{\#}
\newcommand{\z}{\mathbf{z}}
\newcommand{\x}{\mathbf{x}}
\newcommand{\y}{\mathbf{y}}
\renewcommand{\u}{\mathbf{u}}
\newcommand{\bfb}{\mathbf{b}}
\newcommand{\dL}{\mathrm{d}}
\newcommand{\bfc}{\mathbf{c}}
\newcommand{\bfa}{\mathbf{a}}
\newcommand{\bfd}{\mathbf{d}}
\newcommand{\bfe}{\mathbf{e}}
\newcommand{\bff}{\mathbf{f}}
\newcommand{\one}{\mathbb{1}}

\DeclareMathOperator{\Adj}{Adj}
\DeclareMathOperator{\Aut}{Aut}
\newcommand{\Z}{\mathbb{Z}}
\newcommand{\Zprim}{\mathbb{Z}_{\mathrm{prim}}}
\newcommand{\N}{\mathbb{N}}
\newcommand{\R}{\mathbb{R}}

\newcommand{\twosum}[2]{\sum_{\substack{#1\\#2}}}

\newcommand{\beql}[1]{\begin{equation}\label{#1}}
\newcommand{\eeq}{\end{equation}}

\newcommand{\modd}[1]{\; ( \text{mod} \; #1)}
\newtheorem{theorem}{Theorem}
\newtheorem{lemma}{Lemma}

\begin{document}

\title{The Distribution of Rational Points on Conics} 
\author{D.R.\ Heath-Brown}

\begin{abstract}
We examine the counting function for rational points on conics, and show how the point where the asymptotic behaviour begins depends on the size of the smallest zero.

\end{abstract}

\address{Mathematical Institute\\
Radcliffe Observatory Quarter\\ Woodstock Road\\ Oxford\\ OX2 6GG}
\email{rhb@maths.ox.ac.uk}

\subjclass[2010]{11D45 (11D09 11E08 14G05)}

\maketitle

\date{}

\section{Introduction}

Let $q(x_1,x_2,x_3)=q(\x)\in\Z[x_1,x_2,x_3]$ be a quadratic form, and
write
\[\Zprim^n=\{(x_1,\ldots,x_n)\in\Z^n-\{\mathbf{0}\}:
\,\mathrm{g.c.d.}(x_1,\ldots,x_n)=1\}.\]
The
purpose of this paper is to investigate the behaviour as
$B\rightarrow\infty$ of the counting function
\[N(B)=N(B;q)=
\card\{\x\in\Zprim^3:\,q(\x)=0,\,\max\{|x_1|,|x_2|,|x_3|\}\le B\},\]
and of its weighted form
\[N(B,w)=N(B,w;q)=\sum_{\substack{\x\in\Zprim^3\\ q(\x)=0}}w(B^{-1}\x).\]
Here we take $w:\R^3\to\R$ to be infinitely differentiable, with
compact support.  With this notation, the conic $q=0$ has
$\tfrac12 N(B)$ rational points of height at most $B$.

Provided $q$ is isotropic over $\mathbb{Q}$ (in other words, if
$q(\x)=0$ has at least one non-zero integral solution), one has
\[N(B)\sim \tfrac12\sigma_{\infty}\mathfrak{S}(q)B\;\;\;
\mbox{as} \;\;B\to\infty\]
where $\sigma_{\infty}>0$ is the real density of solutions, and
$\mathfrak{S}(q)>0$ may be given
explicitly in terms of the usual product of local densities. (The
factor $\tfrac12$ is the ``alpha constant" in Peyre's terminology \cite{Peyre})
Indeed one has
\beql{nbb}
N(B,w)=\tfrac12\sigma_{\infty}(q;w)\mathfrak{S}(q)B
+O_{q,w}(B\exp\{-c\sqrt{\log B}\}),
\eeq
for some absolute constant $c>0$.   These
results follow from work of the author \cite[Corollary 2]{circle}
We stress that the error term in (\ref{nbb}) contains an unspecified
dependence on
$q$. Our main aim in this paper is to obtain a good explicit
dependence, so as to show how large $B$ has to be, in terms of $q$,
before one sees the true asymptotics for $N(B)$.

In order to see the phenomena that $N(B)$ can display we present a
numerical example.  Let $q_0$ be the form
\beql{q0d}
q_0(\x)=-61x_1^2-22x_1x_3-38x_2^2+99x_2x_3+39x_3^2.
\eeq
Then the following graph shows values of $N(B;q_0)$ for $B\le10000$.
10000.

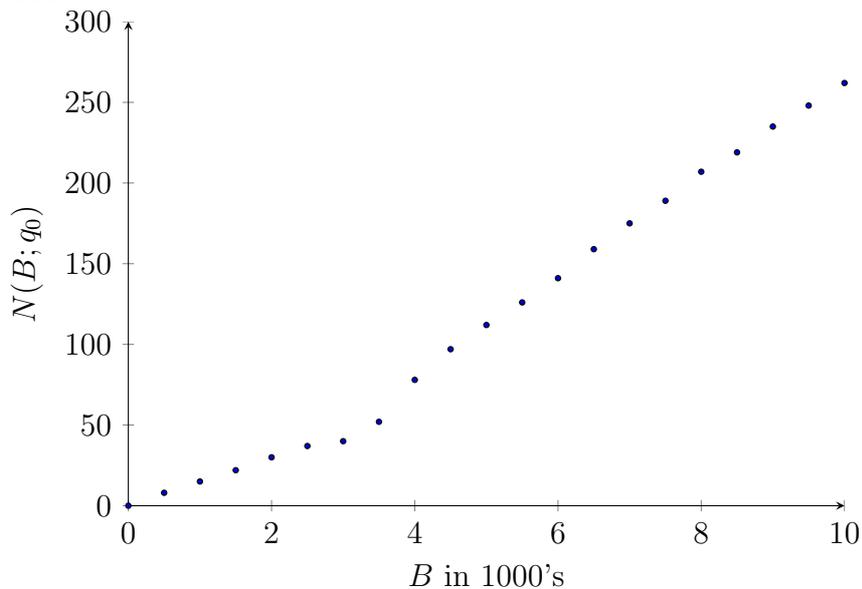
\begin{tikzpicture}
\begin{axis}[
axis lines = left,
xlabel={$B$ in 1000's},
ylabel={$N(B;q_0)$}, 
ymin=0, ymax=300, xmin=0, xmax=10,
 xtick={0,2,4,6,8,10},
  ytick={0,50,100,150,200,250,300},
]
  \addplot+[only marks, black, mark=*, mark size=1pt,]
  coordinates { 
    (0,0) (0.5,8) (1,15) (1.5,22) (2,30) (2.5,37) (3,40) (3.5,52)  (4,78)
    (4.5,97) (5,112) (5.5,126)  (6,141) (6.5,159)
  (7,175) (7.5,189) (8, 207) (8.5,219) (9,235) (9.5,248) (10,262) 
  };
\end{axis}
\end{tikzpicture}

%\includepdf[pages=-]{conic3}
The graph appears linear from about $B=6000$ onwards, but there is a
surprising kink around $B=3500$.  Indeed for $B\le 2500$ the graph
seems linear, but with a smaller gradient than for the range
$B\ge 6000$.  It is this strange behaviour that we aim to explain ---
see the discussion after theorem \ref{t4}.
\bigskip

We begin by introducing some notation and terminology. In general we
will want to allow our form $q$ to have odd cross-terms. We
therefore write it in the asymmetric shape
\[q(\x)=\sum_{1\le i\le j\le 3}q_{ij}x_ix_j,\]
associate with $q$ the matrix
\beql{QD}
Q=\left(\begin{array}{rrr} 2q_{11} &q_{12} &q_{13}\\
  q_{12} & 2q_{22} &q_{23} \\ q_{13} &q_{12} & 2q_{33} \end{array}\right).
  \eeq
Moreover, we define the determinant, somewhat
unconventionally, by
\[\Delta=\Delta(q)=\tfrac{1}{2}\det(Q).\]
Thus $\Delta\in\Z$ for any integral form, and
\[\Delta\big(q(M\x)\big)=\det(M)^2\Delta(q)\]
for any $3\times 3$ matrix $M$. By changing the sign of $q$ if necessary we can
arrange that $\Delta(q)\ge 0$.  We
recall that $q$ is said to be primitive if the coefficients $q_{ij}$
have no common factor.
With this notation our first result is the following.
\begin{theorem}\label{t1}
Let $q$ be a primitive integral isotropic form
  with $\Delta>0$.
Then there is a positive integer $K\le\tau(\Delta)$, and there are
nonsingular $3\times 3$ integer matrices $M_1,\ldots,M_K$  having the
following properties.
\begin{enumerate}
\item[(i)] If $\Delta$ is square-free then $K=\tau(\Delta)$.
  \item[(ii)] The determinant $\det(M_k)$ is a positive divisor of $\Delta$.
\item[(iii)] For any
primitive integral solution $\x$ of the equation $q(\x)=0$, there is a
unique index $k$ such that $\x\in M_k(\Z^3)$.
\item[(iv)]
For each $k$ there is a corresponding $D_k\in\N$ such that
\beql{fun}
q(M_k\x)=D_k(x_1x_3-x_2^2),
\eeq
identically in $\x$.
\item[(v)] 
  We have
\beql{MD}
\Delta\det(M_k)^2=D_k^3,
\eeq
so that $D_k\mid \Delta$ and $\det(M_k)\mid D_k$. Moreover
$D_k\mid\det(M_k)^2$.
\item[(vi)] A prime $p$ can divide $\Delta\det(M_k)^{-1}$ only if
  $v_p(\Delta)\ge 4$.
  \item[(vii)] If $\Delta$ is cube-free then for every index
    $k\le\tau(\Delta)$ the set $M_k(\Z^3)$ contains a primitive zero
    of $q$.
\end{enumerate}
\end{theorem}
Here $\tau(\ldots)$ is the usual divisor function, and $v_p(\Delta)$
is the $p$-adic valuation. In addition to this
notation we will find it convenient to write $J(\x)$ for the quadratic form
$x_1x_3-x_2^2$, so that $q(M_k\x)=D_kJ(\x)$.

The theorem shows that we can partition the primitive integer zeros of
$q$ into $K$ classes $\mathcal{C}_1,\ldots,\mathcal{C}_K$,
corresponding to the different matrices $M_k$.  Specifically, we
define
\[\cl{C}_k=\{\x\in\Zprim^3\cap M_k(\Z^3):\,q(\x)=0\}.\]
Moreover, since the
primitive integer zeros of $J$ are given exactly twice each by
$\pm(u_1^2,u_1u_2,u_2^2)$ , the theorem shows that we can
produce the primitive integer solutions of  $q(\x)=0$ exactly twice
each as 
\[\x=\pm M_k(u_1^2,u_1u_2,u_2^2).\]
Here $(u_1,u_2)$ 
must be primitive if $\x$
is, but unfortunately it is not true that $\x$ is primitive whenever
$u_1$ and $u_2$ are coprime.

Part (vi) of the theorem shows that if $\Delta$ has no
fourth-power divisors then we have $\det(M_k)=D_k=\Delta$ for every
index $k$.  In what follows, it may help the reader if they first
restrict attention to this simplified case.

Part (vii) of the theorem shows
that if $\Delta$ is cube-free then each of the
classes $\cl{C}_k$ is non-empty.  For other values of $\Delta$
we may discard any values
of $k$ for which $\cl{C}_k$ is empty, without affecting the claims in
the theorem.  Thus we will suppose in what follows that each class
$\cl{C}_k$ is non-empty.

The form $q_0$ given by (\ref{q0d}) has $\Delta(q_0)=977861=p_0$, say,
which is prime.  Hence part (i) of the theorem shows that $K=2$, and
parts (v) and (vi) yield
\[\det(M_1)=D_1=\det(M_2)=D_2=p_0.\]
In fact we may take
\beql{Ms}
M_1=\left(\begin{array}{rrr} 1& -45 & 3426 \\ 0 & 100& 3339 \\ 
-1& -54 & 3047 \end{array}\right) \;\;\mbox{and}\;\;
M_2=\left(\begin{array}{rrr} 39 & -21 & -98 \\ 0 & -100 & -1 \\
61  & 122 & 99 \end{array}\right).
\eeq
Indeed
$q_0(\x)=L_1(\x)L_2(\x)-p_0L_3(\x)^2$ with
\[L_1(\x)=100x_1+99x_2+100x_3,\;\; L_2(\x)=9778x_1+9877x_2+9779x_3,\]
and
\[L_3(\x)=x_1+x_2+x_3,\]
and it turns out that the two classes are
\[\cl{C}_1=\{\x\in\Zprim^3:\, q(\x)=0\;\mbox{and}\; p_0\mid L_1(\x)\}\]
and
\[\cl{C}_2=\{\x\in\Zprim^3:\, q(\x)=0\;\mbox{and}\; p_0\mid L_2(\x)\}.\]
\bigskip

In order to use Theorem \ref{t1} for quantitative results we will need
information on the size of the entries in $M_k$.  However for each $k$
there are infinitely many choices for $M_k$, since the automorphism
group
\[\Aut(J)=\Aut_{\Z}(J)=\{U\in{\rm M}_3(\Z):J(M\x)=J(\x)\}\]
is infinite.  (Note that $\det(U)=\pm 1$ for every $U\in\Aut(J)$.  Thus
$U^{-1}$ will be automatically be integral.) Our next result shows that we can
always make a good choice for $M_k$. We will write $||\x||$ for the
$L^2$-norm of the vector $\x$, and
$||q||$ for the $L^2$-norm of the coefficients of the
    matrix $Q$ of $q$, as given by (\ref{QD}). Specifically, we have
    \[||q||=||Q||=
    \{4q_{11}^2+2q_{12}^2+2q_{13}^2+4q_{22}^2+2q_{23}^2+4q_{33}^2\}^{1/2}.\]
  \begin{theorem}\label{t2}
    In Theorem \ref{t1} we may choose $M_k$ so that if $M_k^{-1}$ has
    rows $\b{r}_1,\b{r}_2,\b{r}_3$, then
    \beql{t2crE}
    ||\b{r}_1||\cdot||\b{r}_3||\le 9D_k^{-1}||q||
    \;\;\;\mbox{and}\;\;\; ||\b{r}_2||^2\le 10D_k^{-1}||q||.
    \eeq
Moreover if $M_k$ has columns $\b{c}_1,\b{c}_2,\b{c}_3$
we will have $||\b{c}_1||\le||\b{c}_3||$,
\beql{t213}
||\b{c}_1||\cdot||\b{c}_3||\le 90 \det(M_k)^2D_k^{-2}||q||^2
\le 90 ||q||^2
\eeq
and
\beql{t22}
||\b{c}_2||\le 9\det(M_k)D_k^{-1}||q|| \le 9||q||.
\eeq
   \end{theorem}
The constants can certainly be improved, but for our purposes it
suffices to know that there is at least one set of numerical values that is
valid. From now on we will assume that the matrices $M_k$ are as
described in Theorem \ref{t2}. For the form
(\ref{q0d}) we have
\[||q_0||=\sqrt{39872}=199.679\ldots,\]
and one sees that the columns of the matrices (\ref{Ms}) amply fulfil the
conditions above.

  Since $q(\b{c}_1)=q(\b{c}_3)=0$, we find that for any isotropic form $q$
  and any index $k$ there are two linearly independent zeros
  $\b{c}_1,\b{c}_3\in M_k(\Z^3)$ with
  $||\b{c}_1||\cdot||\b{c}_3||\le 90||q||^2$; in particular there is at least
  one vector in $\cl{C}_k$ of length at most $10||q||$. (The reader
  should note that $\b{c}_1$ and $\b{c}_3$ need not be primitive, while
$\cl{C}_k$ is defined as the set of primitive zeros in $M_k(\Z^3)$.)
  Thus Theorem \ref{t2} recovers (in the case of ternary forms) the
  results of Davenport \cite[Theorem 1]{Dav} and Cassels \cite{Cas},
  which were weaker in as much as they referred only to the complete
  set of zeros of $q$, rather than individual classes $\cl{C}_k$. 

Our next result, which is rather easy, explains how $\b{c}_1$ and
$\b{c}_3$ are related to the smallest and second smallest zeros of $q$
in $\cl{C}_k$.  It is phrased in terms of the quantity
\[\rho=\rho(q)=\frac{||q||^3}{\Delta},\]
which we will refer to as the ``aspect ratio" of $q$.  We will see that
$\rho\ge 2$ in all cases, and we may expect that $\rho(q)\approx 1$ for
``typical" forms $q$.  We therefore think of forms with large aspect
ratio as having untypically small determinant. For the form
(\ref{q0d}) we have $\rho(q_0)=8.141\ldots$.

\begin{theorem}\label{t5}
We have $\rho(q)\ge 2$ for any $q$.

Let $\z_1$ be an element of $\cl{C}_k$ of minimal length, and let $\z_2$
be an element of $\cl{C}_k$ of minimal length subject to the condition
that $\z_2\not=\pm\z_1$. Then
\beql{z1z2}
||\z_1||\cdot||\z_2||\ge D_k/||q||.
\eeq
If
\[||\b{c}_1||< \rho^{-1/2}\det(M_k)D_k^{-1}||q||,\]
then $\b{c}_1$ must be a scalar multiple of the shortest vector $\z_1$.
In general
\[||\b{c}_1||\le 90\rho||\z_1||\;\;\;\mbox{and}\;\;\;
||\b{c}_3||\le 90\rho||\z_2||.\]
  \end{theorem}

When $\rho\ll 1$, as we usually expect, we may interpret Theorem
\ref{t5} as saying that the lengths of $\b{c}_1$ and $\b{c}_3$ are
within a constant factor of the shortest possible lengths, namely
$||\z_1||$ and $||\z_2||$.  Moreover, suppose we write $c$ for the
constant $c=(91\rho)^{-1/2}$ and  
 \[\ell=\sqrt{90}\det(M_k)D_k^{-1}||q||\]
for the maximum length for $\b{c}_1$ permitted by (\ref{t213}). Then
whenever $\b{c}_1$ has length at most $c\ell$, the vector $\b{c}_1$
must actually  
be the minimal zero $\z_1$, or a scalar multiple of it. For the form $q_0$,
the first column is a minimal zero, so that any other zero in
$\mathcal{C}_1$ must have length
at least $2^{-1/2}p_0/||q_0||=3462.805\ldots$ The third column of
$M_1$ gives a zero of length $5671.913\ldots$. In contrast, the first
and third columns of $M_2$ give relatively small zeros in $\cl{C}_2$.
  
  In fact Theorem \ref{t2} follows from the following more general
  result.
\begin{theorem}\label{t3}
Suppose $A$ is a $3\times 3$ integer matrix, and that $J(A\x)=q(\x)$.
Then there is a matrix $U\in \Aut(J)$ such that the rows
$\bfa,\bfb,\bfc$ of $UA$ satisfy
\[||\bfa||\cdot||\bfc||\le 9||q||\;\;\;\mbox{and}
\;\;\; ||\bfb||^2\le 10||q||.\]
\end{theorem}

We may use the previous theorems to count primitive zeros of $q(\x)$.
Our eventual aim is to give a sharp explicit version of the asymptotic
formula (\ref{nbb}).  We begin by estimating the number of zeros in
each of the classes $\mathcal{C}_k$ with height at most $B$, using the
counting function
\[N(B;\mathcal{C}_k)=\sum_{\x\in\mathcal{C}_k}w(B^{-1}\x).\]
Previously we had said that $:\R^3\to\R$ should be infinitely
differentiable, with compact support.  We shall now be more specific
and require that $w(\x)=0$ whenever $||\x||>1$.
Since it is possible that $w(\x)$ might vanish on
the zero locus of $q$ we introduce a second weight function $w_0(x)$
defined  as
\beql{w0d}
w_0(\x)=\left\{\begin{array}{cc}
\exp\left\{-\frac{1}{1-||\x|^{2}}\right\}, &
||\x||<1,\\ 0, & \mbox{otherwise.}\end{array}\right.
\eeq
This has the properties required for $w$ itself, and its support
includes non-trivial points of the conic $q(\x)=0$. We now define the
real density of points on the conic $q=0$,relative to the weight $w$
by setting
  \beql{sid}
  \sigma_{\infty}(q;w)=\lim_{T\to\infty}\int_{\R^3}w(\y)K_T(q(\y))dy_1dy_2dy_3,
 \eeq
  with
  \beql{KTD}
  K_T(t)=T\max\{1-T|t|\,,\,0\}.
  \eeq
  This coincides with the constant occuring in (\ref{nbb}), see
  Theorem 3 of \cite{circle} which has a mild variant of (\ref{sid}).
  We shall show in Lemma \ref{MT} that the above limit does indeed exist.
The following asymptotic formula for $N(B;\mathcal{C}_k)$ then holds.
\begin{theorem}\label{t4}
For each class $\cl{C}_k$ there is a square-free divisor $\Delta_1\Delta_2$
of $\Delta$ such that $p\mid\Delta_1$ for every prime for which
$p||\Delta$, and such that
  \begin{eqnarray*}
    N(B;\mathcal{C}_k)&=&\frac{\Delta^{1/2}}{2D_k^{1/2}}
    \sigma_{\infty}(q,w)\kappa B\left\{1+O_w\left(\psi(B)
    \left(\frac{||\b{z}_2||}{B}\right)^{1/4}\right)\right\}\\
  &&\hspace{1cm}{}+O_w(1),
  \end{eqnarray*}
  with
  \[\kappa=\frac{6}{\pi^2}\prod_{p\mid \Delta_1}\frac{1}{1+p^{-1}}
\prod_{p\mid \Delta_2}\frac{1-p^{-1}}{1+p^{-1}},\]
  and
 \[\psi(B)=4^{\omega(\Delta)}\rho\,\frac{\sigma_{\infty}(q,w_0)}
         {\sigma_{\infty}(q,w)}\log B.\]
 Here $\z_2$ is the second smallest element of $\cl{C}_k$, as
 described precisely in Theorem \ref{t5}.
\end{theorem}

A number of comments should be made here.  Firstly, in interpreting
the theorem one should think of the factor $\psi(B)$ as being roughly of
order 1, or more generally as not being too large. We will see in
Lemma \ref{cqE} that $\sigma_{\infty}(q,w)\ll\sigma_{\infty}(q,w_0)$
when $\sup|w|\le 1$. However we have no estimate in the reverse direction
since it is possible that the zero locus of $q$ only just enters the
support of $w$, making $\sigma_{\infty}(q,w)$ small. None the less it
is reasonable to think that $\sigma_{\infty}(q,w_0)\ll\sigma_{\infty}(q,w)$
in most cases of interest.

Viewing $\psi(B)$ as being small we may
interpret the theorem as giving a linear asymptotic formula for 
$N(B;\mathcal{C}_k)$ which takes effect when $B$ is not much larger
than $||\b{z}_2||$. Indeed one can easily show that the error term
$O_w(1)$ is insignificant when $B\ge\rho\,||\b{z}_2||$. Of course when
$B<||\b{z}_2||$ the function $N(B;\mathcal{C}_k)$ counts at most the
zeros $\pm\b{z}_1$.  Thus $N(B;\mathcal{C}_k)$ is $O_w(1)$
from $B=1$ to $B=||\b{z}_2||$, and then begins to display its typical
linear growth.

When $\Delta$ is square-free we have $D_k=\Delta$ for every $k$, by
parts (v) and (vi) of Theorem \ref{t1}.  Moreover we will have
$\Delta_1=\Delta$ and $\Delta_2=1$, so that
\[\kappa=\frac{6}{\pi^2}\prod_{p\mid \Delta}\frac{1}{1+p^{-1}},\]
for each index $k$. Thus when $\Delta$ is square-free the leading
constant in Theorem \ref{t4} is the same for each value of $k$, but
the point at which linear growth begins is potentially different.

We are now in a position to explain the observed kink in our graph
of $N(B)$ for the quadratic $q_0$. The correspondence between $N(B)$
and the counting functions $N(B;\cl{C}_k)$ is not precise since the
former is defined using the condition $||\x||_{\infty}\le B$ while the
latter use $||\x||=||\x||_2$. For the class $\cl{C}_1$ we may take
$\z_1=(1,0,-1)$.  The zero of second smallest sup-norm in $\cl{C}_1$
is $(3426,3339,3047)$ whence $N(B;\cl{C}_1)=2$ for $1\le B<3426$.
However as soon as $B$ is somewhat larger than $3500$ we will have
$N(B;\cl{C}_1)\sim cB$,
for a certain constant $c>0$.  For $\cl{C}_2$ the two zeros of
smallest sup-norm are $(39,0,61)$ and $(-98,-1,99)$ (or
$(-38,-99,38)$, which has the same sup-norm) so that we will have
$N(B;\cl{C}_2)\sim cB$, as soon as $B$ is somewhat larger than a
few hundred, with
the same constant $c$. Thus the initial section of the graph for
$N(B,q_0)$, up to $B=3500$ or so, reflects the range in which
$N(B;\cl{C}_1)=2$ but $N(B;\cl{C}_2)$ is already growing like $cB$,
and the later values of $B$ are in the range where both
$N(B;\cl{C}_1)$ and $N(B;\cl{C}_2)$ are growing like $cB$.

Some remarks on the shape of $\psi(B)$ are also in order. It would be
interesting to know to what extent the various factors involved could
be reduced, or indeed removed. Although this seems possible to some
extent, we hope that the present form of $\psi(B)$ will be sufficient
for applications.

We can produce an asymptotic formula for $N(B,w;q)$ by summing up the
formulae for $N(B;\mathcal{C}_k)$. Since $||\b{z}_2||\le ||\b{c}_3||$
and $||\b{z}_1||\le ||\b{c}_1||$ for each index $k$ the inequality
(\ref{t213}) yields
\[||\b{z}_2||\le ||\b{c}_3||\ll ||q||^2/||\b{c}_1||\le
||q||^2/||\b{z}_1||.\]
Thus Theorem \ref{t4} has the following immediate corollary, in light
of part (i) of Theorem \ref{t1}.
\begin{theorem}\label{cor}
Let $\z_0$ be a non-trivial integer zero of $q$ with $||\z_0||$
minimal.  Then
\begin{eqnarray*}
N(B,w;q)&=&\tfrac12\sigma_{\infty}(q;w)\mathfrak{S}(q)B
\left\{1+O_w\left(\psi(B)
\left(\frac{||q||^2}{||\b{z}_0||B}\right)^{1/4}\right)\right\}\\
&&\hspace{1cm}{}+O_w(\tau(\Delta)),
\end{eqnarray*}
with $\psi(B)$ as in Theorem \ref{t4}.
\end{theorem}

This is the promised improvement of (\ref{nbb}), with a good explicit 
dependence on $q$. It produces a linear asymptotic growth as soon as
$B$ is a little larger than $||q||^2/||\z_0||$.  Since $||\z_0||$ is typically 
of order around $||q||$ this is essentially best possible. We should also 
comment on the quality of the error term, which has a power saving in $B$.
In (\ref{nbb}) there is a saving of order $\exp\{-c\sqrt{\log B}\}$, which 
has its origins in the error term for the Prime Number Theorem.  Thus 
one could replace $\sqrt{\log B}$ in the exponent by some slightly larger 
power of $\log B$, but one
cannot hope to establish (\ref{nbb}) with a power saving in $B$ by the
methods of \cite{circle}.

The reader may compare our work with that of Sofos \cite{sofos}. The latter gives an asymptotic formula for an unweighted counting function, and has an error term which has a better dependence on $B$ (of order $B^{1/2}\log B$) and an explicit dependence on $q$, though a much weaker one.

In future work we plan to apply Theorem \ref{cor} to count rational
points on certain varieties that can be fibred into conics.  Indeed
such applications provide the natural motivation for the present
paper.  In work in preparation (jointly with Dan Loughran) we look at the counting function for Del Pezzo surfaces of degree 5, in the case where there is a conic fibration.
Another example, which we plan to examine in due course, is the variety
$V\in\mathbb{P}^2\times\mathbb{P}^2$ cut out by the equation
\[X_0Y_0^2+X_1Y_1^2+X_2Y_2^2=0,\]
in which a rational point $P$ represented by a pair of primitive integer
vectors $(\x,\y)$ has height $h(P)=||\x||_{\infty}^2||\y||_{\infty}$.
Both these examples require the full strength of Theorem \ref{cor}.

\section{Proof of Theorem \ref{t1}}

We begin with a result that will allow us to work with matrices over
$\Z/mZ$, rather than $\Z$.
\begin{lemma}\label{wa}
  Let $M$ be an $n\times n$ integer matrix, with determinant coprime
  to some positive integer $r$. Then there is a matrix
  $M'\equiv M\modd{r}$ with prime determinant.
  Moreover, if $\det(M)\equiv 1\modd{r}$
  there is an $M''\equiv M\modd{r}$ in $\mathrm{SL}_n(\Z)$.
\end{lemma}
\begin{proof}
  We can write $M$ in Smith Normal Form as $M=UDV$ with
  $U,V\in\mathrm{SL}_n(\Z)$ and $D$ diagonal.  One then sees that it
  suffices to prove the lemma when $M$ is diagonal, which we do by
  induction on $n$. The case $n=1$ is immediate, by Dirichlet's
  Theorem. If the result is true for matrices of size $n-1$, and
  \[M=\mathrm{Diag}(m_1,\ldots,m_n)=\left(\begin{array}{c|c} M_0 &
      0\\ \hline 0 & m_n\end{array}\right),\]
      say, then $\det(M_0)$ will be coprime to $r$ so that
  $M_0\equiv M_0'\modd{r}$ with $\det(M_0')$ prime. It follows that
      we may write $M_0'$ in Smith Normal Form as $U_0D_0V_0$, with
      $D_0=\mathrm{Diag}(1,\ldots,1,p)$ say. Thus
      $M\equiv M_1\modd{r}$ with
      \[M_1=\left(\begin{array}{c|c} M'_0 &
        0\\ \hline 
        0 & m_n\end{array}\right)=U_1\mathrm{Diag}(1,\ldots,1,p,m_n)V_1,\]
      where
      \[U_1=\left(\begin{array}{c|c} U_0 & 0\\ \hline 0 &
          1\end{array}\right),\]
          and similarly for $V_1$. To complete the induction step it
          remains to show that the lemma holds for the matrix
          $\mathrm{Diag}(p,m_n)$. However
          \[\mathrm{Diag}(p,m_n)\equiv\left(\begin{array}{cc}
              p & sr\\ r & m_n+tr\end{array}\right)\modd{r}\]
              and the matrix on the right has determinant
              $pm_n+tpr-sr^2$. Since $pm_n$ will be coprime to $r$ we
              can make this determinant prime by taking $t=0$ and
              choosing $s$ suitably. Moreover, if $pm_n=1+kr$, we can
              make the determinant equal to 1 by choosing $s$ and $t$
              so that $sr-tp=k$.  This completes the induction
              argument.
              \end{proof}

Our next result describes the reduction of ternary forms modulo
a prime $p$ and its powers.  We do not assume that $p$ is odd.
\begin{lemma}\label{rp}
  Let $p$ be prime and let $q(\x)$ be an integral ternary quadratic
  form, not divisible by $p$ but with
  $p^e||\Delta(q)$ for some exponent $e\ge 1$.
  Then there is a matrix $M\in\mathrm{SL}_3(\Z)$ such that one of the
  following holds.
  \begin{enumerate}
  \item[(i)] $q(M\x)\equiv \kappa x_3^2\modd{p}$, with $p\nmid \kappa$;
  \item[(ii)] $q(M\x)\equiv x_1x_2+\kappa p^ex_3^2\modd{p^{e+1}}$ for some
    integer $\kappa$ coprime to $p$;
 \item[(iii)] $q(M\x)\equiv q_1(x_1,x_2)+\kappa p^ex_3^2\modd{p^{e+1}}$ for some
    integer $\kappa$ coprime to $p$, with $q_1$ irreducible modulo $p$.    
  \end{enumerate}
  \end{lemma}

\begin{proof}
    In view of Lemma \ref{wa} it suffices to find a suitable $p$-adic matrix
$M\in\mathrm{SL}_3(\Z_p)$. When $p$ is odd we can diagonalize over
$\Z_p$ to give $Ax_1^2+Bx_2^2+Cx_3^2$, say.  Since $q$ has determinant
divisible by $p$, but does not vanish modulo $p$ we see 
that either case (i) of the lemma holds, or that we may take
$p\nmid AB$ and $p^e||C$. We then have case (ii) if $-AB$ is
a quadratic residue of $p$, and case (iii) otherwise.

For $p=2$ we consider the reduction of $q$ over $\mathbb{F}_2$.  Since
$2\mid\Delta$ we find that $q(\x)$ is equivalent to one of $x_3^2$, or
$x_1x_2$, or $x_1^2+x_1x_2+x_2^2$ over $\mathbb{F}_2$, via a matrix in 
$\mathrm{SL}_3(\mathbb{F}_2)$. (This can be shown by considering all
possible forms $q$ modulo 2, if necessary.) The first case leads
immediately to case (i) of the lemma. In the remaining cases, Lemma
\ref{wa} shows that $q$ is equivalent to
$\tilde{q}(x_1,x_2)+\ell(x_1,x_2)x_3+\mu x_3^2$ over $\Z_2$, where
$\ell(x_1,x_2)$ is a linear form, and
$\tilde{q}(x_1,x_2)\equiv x_1x_2$ or
$x_1^2+x_1x_2+x_2^2\modd{2}$. Replacing $x_1$ and $x_2$ by
$x_1-\xi_1x_3$ and $x_2-\xi_2x_3$ respectively eliminates the term
$\ell(x_1,x_2)x_3$ provided that
\[\ell(x_1,x_2)=\xi_1\frac{\partial\tilde{q}(x_1,x_2)}{\partial x_1}
+\xi_2\frac{\partial\tilde{q}(x_1,x_2)}{\partial x_2}.\]
Suitable $\xi_1,\xi_2\in\Z_2$ can always be found, since the linear
forms $\partial\tilde{q}/\partial x_1$ and
$\partial\tilde{q}/\partial x_2$ are congruent modulo 2 to $x_2$ and
$x_1$ respectively. We then conclude that $q$ is equivalent to
$\tilde{q}(x_1,x_2)+\mu' x_3^2$ over $\Z_2$.  Computing the
determinant of this we find that $2^e||\mu'$. When
\[\tilde{q}(x_1,x_2)\equiv x_1^2+x_1x_2+x_2^2\modd{2}\]
we obtain case
(iii) of the lemma.  Finally, if $\tilde{q}(x_1,x_2)\equiv x_1x_2\modd{2}$
we see from Hensel's Lemma that $\tilde{q}(x_1,x_2)$ must factor over
$\Z_2$, and a further unimodular change of variables leads to case
(ii) of the lemma.
\end{proof}

We next have the following lemma, which shows how we remove powers of
$p$ from $\Delta(q)$.
\begin{lemma}\label{L1}
  Suppose that $q(\x)$ is an integral isotropic ternary
  quadratic form, not necessarily primitive, and that
$p^e||\Delta(q)\not=0$. Then there is a positive integer $K\le e+1$
  such that $K=2$ when $e=1$, and there are $3\times 3$ integer matrices
$R_1,\ldots,R_K$ with determinants
$\det(R_k)=p^{\mu_k}$, such that the following properties hold.
Firstly, $\mu_k\le e$ is a non-negative integer with
$\mu_k\equiv e\modd{3}$ for each $k\le K$.  Secondly, the
form
\beql{qk}
p^{-(e+2\mu_k)/3}q(R_k\x)
\eeq
has integer coefficients and has determinant $p^{-e}\Delta(q)$.
Thirdly, if $q(\x)$ vanishes for some primitive $\x\in\Z^3$,
then there is exactly one index $k\le K$ for which
$R_k^{-1}\x\in\Z^3$.
\end{lemma}

\begin{proof}
 Clearly the form (\ref{qk}) has determinant
  $p^{-e}\Delta(q)$. The proof of the lemma
  will be by induction on $e$. When $e=0$ we have $K=1$ and
  $R_1$ can be taken to be the identity. To handle the induction step
  we assume that the lemma holds for exponents
  strictly less than $e$. Suppose firstly that the form $q$ is
  identically divisible by $p$, so that $e\ge 3$. Write $q'(\x)=p^{-1}q(\x)$, 
  whence $p^{e-3}||\Delta(q')$.  By the induction assumption we have
  matrices $R'_1,\ldots,R'_J$ with $J\le (e-3)+1=e-2$, and
  exponents $\mu'_k\le e-3\le e$. We now claim that we can take
  $K=J\le e+1$ and $R_k=R'_k$ for every index $k$, so that
  $\mu_k=\mu'_k$. In the first place we have
  \[\mu_k=\mu'_k\equiv e-3\equiv e\modd{3}.  \]
  Secondly,
  \[p^{-(e+2\mu_k)}q(R_k\x)=p^{-\{(e-3)+\mu'_k\}/3}q'(R'_k\x)\]
    which is an integral form. Thirdly, if
    $q(\x)=0$ for some primitive $\x\in\Z^3$, then $q'(\x)=0$, whence
    there is exactly one index for which ${R'_k}^{-1}\x$ is
    integral. Thus there is exactly one index for which $R_k^{-1}\x$ is
    integral.

When $q(\x)$ is not identically divisible by $p$ we apply Lemma
\ref{rp}, and consider separately the three possible cases. Suppose
firstly that $q(M\x)\equiv \kappa x_3^2\modd{p}$, with $p\nmid \kappa$.
In this case we must have $e\ge 2$.
Then if $M'=\mathrm{Diag}(1,1,p)$ the form $q'(\x)=p^{-1}q(MM'\x)$
    will be integral, with determinant $p^{-1}\Delta(q)$. Moreover it is
    still isotropic, so that we may apply the induction hypothesis
    to $q'$, with
    $p^{e-1}||\Delta(q')$. This produces matrices $R'_1,\ldots,R'_J$
    with $J\le e$, and exponents $\mu'_k\le e-1$ such that
    $\det(R'_k)=p^{\mu'_k}$. We now claim that we can take $K=J$ and
    $R_k=MM'R'_k$ in the lemma.  This will have determinant
    $p^{\mu_k}$ with $\mu_k=1+\mu'_k\equiv 1+(e-1)=e\modd{3}$, as
    required. Moreover
    \begin{eqnarray*}
      p^{-(e+2\mu_k)/3}q(R_k\x)&=&p^{-(e+2\mu_k)/3}pq'(R'_k\x)\\
      &=& p^{-(\{e-1\}+2\mu'_k)/3}q'(R'_k\x)
    \end{eqnarray*}
    which is an integral form, by the induction hypothesis. Finally, when
    $q(\x)=0$ with a primitive $\x\in\Z^3$, we set $\y=M^{-1}\x$, so
    that
    \[0=q(M\y)\equiv\kappa y_3^2\modd{p},\]
    with $p\nmid \kappa$. Then
    $p\mid y_3$, whence ${M'}^{-1}\y\in\Z^3$. It follows that the
    vector $\z=(MM')^{-1}\x$
    is integral, and is primitive since $\x=MM'\z$ is
    primitive. Moreover $q'(\z)=p^{-1}q(\x)=0$, whence the induction
    hypothesis shows that there is exactly one $R'_k$ for which
    ${R'_k}^{-1}\z\in\Z^3$. Hence there is exactly one index $k$ such
    that $R_k^{-1}\x\in\Z^3$.  This completes the proof of Lemma
    \ref{L1} when we are in case (i) of Lemma \ref{rp}.

We turn next to case (ii) of Lemma \ref{rp}, in which
\[q(M\x)\equiv x_1x_2+\kappa p^ex_3^2\modd{p^{e+1}}\]
for some integer $\kappa$ coprime to $p$. We claim that we may take
$K=e+1$ in Lemma \ref{L1}, and
\[R_k=M\mathrm{Diag}(p^{k-1},p^{e+1-k},1)\;\;\;
\mbox{for}\;\;\; 1\le k\le K,\]
    so that $\mu_k=e$ for every $k$. With this choice we have
    \[q(R_k\x)\equiv p^e(x_1x_2+\kappa x_3^2)\modd{p^{e+1}},\]
    so that $p^{-e}q(R_k\x)$ is integral.  Suppose now that $q(\x)=0$
    with $\x$ primitive, and write $\y=M^{-1}\x$, so that
    \beql{cg}
    0=q(M\y)\equiv y_1y_2+\kappa p^ey_3^2\modd{p^{e+1}}.
    \eeq
    It follows that $p^e\mid y_1y_2$, whence there is a positive
    integer $k\le e+1$ such that $p^{k-1}\mid y_1$ and
    $p^{e+1-k}\mid y_2$. We then see that $\y$ lies in the image of
    $\mathrm{Diag}(p^{k-1},p^{e+1-k},1)$, so that
    $\x\in R_k(\Z^3)$. Finally if we also have $\x\in R_j(\Z^3)$
    for some $j>k$ then $p^{j-1}\mid y_1$, whence $p^k\mid y_1$. Since
    $p^{e+1-k}\mid y_2$ it would follow firstly that
    $p^{e+1}\mid y_1y_2$, and secondly that $p\mid y_2$, since
    $e+1-k>e+1-j\ge 0$. However when $p^{e+1}\mid y_1y_2$ the
    congruence (\ref{cg}) shows that $p\mid y_3$. We then reach a
    contradiction, since $p$ cannot
    divide $\y$ when $\x$ is primitive.  This completes the proof of Lemma
    \ref{L1} when we are in case (ii) of Lemma \ref{rp}.
    
Finally we examine case (iii) of Lemma \ref{rp}, in which
\[q(M\x)\equiv q_1(x_1,x_2)+\kappa p^ex_3^2\modd{p^{e+1}}\]
for some integer $\kappa$ coprime to $p$, with $q_1$ irreducible
modulo $p$.  One sees that if $q$ is isotropic we must have $e\ge 2$.
The argument is now similar to that for case (i). Let
    $M'=\mathrm{Diag}(p,p,1)$. Then the form
    $p^{-2}q(MM'\x)$ will be integral, with determinant
    $p^{-2}\Delta(q)$. This will have corresponding matrices
    $R'_k$ with determinant $p^{\mu'_k}$, and we may take
    $R_k=MM'R'_k$ with corresponding value $\mu_k=\mu'_k+2$. We
    leave the reader to verify that these fulfil the conditions for
    Lemma \ref{L1}.  This completes the argument.
    \end{proof}

We are now ready to prove Theorem \ref{t1}
\begin{proof}
We will use induction on the number of distinct prime
divisors of $\Delta$.  We therefore begin by considering the case in which
$\Delta=1$. Here we will have $K=1$, and we claim that we may take
$D_1=1$.
Since $q(\x)$ is isotropic there is a primitive integer vector $\z$
such that $q(\z)=0$. We may then construct a unimodular integer matrix $M$ with
first column $\mathbf{z}$. This produces a form $q(M\x)$ equivalent to
$q$ and taking the shape $x_1(ax_2+bx_3)+q_1(x_2,x_3)$. The
coefficients $a$ and $b$ must be coprime, since $\Delta=1$. A
further unimodular transformation involving $x_2$ and $x_3$ produces
$x_1x_3+q_2(x_2,x_3)$, say. Now we replace $x_1$ by
$x_1+Ax_2+Bx_3$ for suitable integers $A,B$ to obtain a form
$x_1x_3+\lambda x_2^2$. Since the determinant is still $\Delta(q)=1$
we see that
$\lambda=-1$. Thus $q$ is transformed into $x_1x_3-x_2^2$
by a unimodular integer matrix, as required.

Now suppose that $p^e||\Delta$. Lemma \ref{L1} produces matrices
$R_1,\ldots,R_K$ with corresponding exponents $\mu_k$, such that the forms
\[q_k(\x):=p^{-(e+2\mu_k)/3}q(R_k\x)\]
have determinant $p^{-e}\Delta$. Our induction hypothesis, applied to
$q_k$, now produces
further matrices $M_{1,k},\ldots,M_{J,k}$ with
$J=J(k)\le\tau(p^{-e}\Delta)$.  Since the index $k$ runs up to $e+1$ at
most, there are at most
\[(e+1)\tau(p^{-e}\Delta)=\tau(\Delta)\]
matrices in total.  Moreover if $e=1$ the index $k$ takes exactly the
two values 1 and 2.
We now claim that the matrices $R_kM_{j,k}$ have the required
properties. Firstly, $K$ is increased by a factor 2 for each prime
factor $p||\Delta$, so that $K=\tau(\Delta)$ if $\Delta$
is square-free.  Secondly,
$\det(M_{j,k})$ divides $p^{-e}\Delta$ by the
induction hypothesis, and since
$\det(R_k)\mid p^e$ it follows that $\det(R_kM_{j,k})\mid \Delta$, as
required. Thirdly, we observe that
\[q(R_kM_{j,k}\x)=p^{(e+2\mu_k)/3}q_k(M_{j,k}\x)=D_{j,k}J(\x)\]
for a suitable integer $D_{j,k}$.

For part (iv), let $q(\x)=0$ for some
primitive $\x\in\Z^3$. Then, according to Lemma \ref{L1}, there is 
an index $k$ for which $R_k^{-1}\x\in\Z^3$. Moreover, if we write
$\y=R_k^{-1}\x$ then $\y$ must be a primitive integer vector, and 
$q_k(\y)=0$.  Then, by the induction assumption there is a choice of
$j$ such that $M_{j,k}^{-1}\y\in\Z^3$.  Thus
$(R_kM_{j,k})^{-1}\x\in\Z^3$. Finally, if we also have
$(R_hM_{i,h})^{-1}\x\in\Z^3$  we may write
$(R_kM_{j,k})^{-1}\x=\mathbf{u}\in\Z^3$ and
$(R_hM_{i,h})^{-1}\x=\mathbf{v}\in\Z^3$.  Then
$R_k^{-1}\x=M_{j,k}\mathbf{u}$ and $R_h^{-1}\x=M_{i,h}\mathbf{v}$ are
both integral.  According to Lemma \ref{L1} we must therefore have
$k=h$. Thus $R_h^{-1}\x=R_k^{-1}\x=\y$, and both
$M_{i,h}^{-1}\y=M_{i,k}^{-1}\y=\mathbf{v}$ and
$M_{j,k}^{-1}\y=\mathbf{u}$ are integral. Our induction hypothesis
then shows that we must have $i=j$, so that there is exactly one
choice of $k$ and $j$ for which $(R_kM_{j,k})^{-1}\x$ lies in
$\Z^3$.

To handle the remaining claims of the theorem we do not use the
induction argument. Given (\ref{fun}), we obtain the relation
\[\Delta\det(M_k)^2=D_{k}^3\]
by taking determinants. Since $\det(M_k)\mid\Delta$ we have
\[\Delta\det(M_k)^2\mid\Delta^3\;\;\;\mbox{and}\;\;\;
\det(M_k)^3\mid\Delta\det(M_k)^2,\]
so that $D_{k}^3\mid\Delta^3$ and
$\det(M_k)^3\mid D_{k}^3$. Next we write
$\x=\Adj(M_k)\y$ in (\ref{fun}) and note that $\Adj(M_k)=\det(M_k)M_k^{-1}$.
This yields
\[\det(M_k)^2q(\y)=D_{k}J\left(\Adj(M_k)\y\right),\]
whence $D_k\mid\det(M_k)^2$, since the form $q$ was assumed to be
primitive. This establishes part (v).  For part (vi)
we see that if $p\mid\Delta$ then $p\mid D_k$, and
hence $p\mid \det(M_k)$.  Moreover if $p^e||\Delta$ and $p^f||\det(M_k)$,
then $3\mid e+2f$, since $\Delta\det(M_k)^2$ is a cube. We then see
that we must $e=f$ whenever $e\le 3$.  Finally, if $\Delta$ is
cube-free then so is $\det(M_k)$, whence the entries of $M_k$ can have no
common factor. Any vector $M_k(u^2,uv,v^2)$ will be a zero of $q$, so
we need to find integers $u,v$ for which $M_k(u^2,uv,v^2)$ is
primitive.  If $p$ is a prime not dividing $\Delta$ then $M_k$ is
invertible modulo $p$ so that $p\nmid M_k(u^2,uv,v^2)$ whenever
$p\nmid(u,v)$. Otherwise $p$ can divide at most two columns of $M_k$.
If $p$ does not divide the first column then $p\nmid M_k(1,0,0)$.
Similarly if $p$ does not divide the third column of $M_k$ then
$p\nmid M_k(0,0,1)$. Finally, if $p$ divides the first and third
columns but not the second, then $p\nmid M_k(1,1,1)$.  It follows, via
the Chinese Remainder Theorem, that if $(u,v)$ lies in a suitable
residue class modulo $\Delta$ then the vector $M_k(u^2,uv,v^2)$ will
be coprime to $\Delta$.  One can now show via the standard arguments
that the set of integer pairs $u,v$ in such a residue class for which
$u$ and $v$ are coprime, will have positive density, given by
\[\Delta^{-2}\prod_{p\nmid\Delta}(1-p^{-2}).\]
We therefore obtain infinitely many pairs $u,v$ for which
$M_k(u^2,uv,v^2)$ is primitive.

This completes the proof of the theorem
\end{proof}

\section{Proof of Theorem \ref{t3}}

We begin with the following informal observation.  If the coefficients
of $A$ are very large compared with those of $q$, then $J(A\x)=q(\x)$ has
coefficients which are much smaller than they might be, so that $J(A\x)$
``nearly vanishes''.  If $A$ has rows
$\mathbf{a},\mathbf{b},\mathbf{c}$, then
\[J(A\x)=(\mathbf{a}.\x)(\mathbf{c}.\x)-(\mathbf{b}.\x)^2,\]
so that $(\mathbf{a}.\x)(\mathbf{c}.\x)$ is approximately equal to
$(\mathbf{b}.\x)^2$.  If in fact they were identically equal, the
linear forms $\mathbf{a}.\x$, $\mathbf{b}.\x$ and $\mathbf{c}.\x$
would have to be proportional, and so the vectors
$\mathbf{a}$, $\mathbf{b}$ and $\mathbf{c}$ would also be
proportional.

Our next lemma confirms this, in a quantitative way.
\begin{lemma}\label{ald}
  Suppose $A$ has rows $\mathbf{a},\mathbf{b},\mathbf{c}$, and that
  $||\mathbf{a}||\le ||\mathbf{c}||$.  Write
  $\bfa=\lambda\bfc+\bfd$ and $\bfb=\mu\bfc+\bfe$, where $\bfd$ and
  $\bfe$ are orthogonal to $\bfc$.  Then if $J(A\x)=q(\x)$ we have
  \begin{enumerate}
  \item[(i)] $||\bfe||\le 2^{-1/2}||q||^{1/2}$;
  \item[(ii)] $|\lambda-\mu^2|\le||q||/(2||\bfc||^2)$; 
  \item[(iii)] $||\bfd-2\mu\bfe||\le||q||/||\bfc||$; and
    \item[(iv)] $|\lambda|\le 1$.    
  \end{enumerate}
\end{lemma}
The reader should note that the bounds (i), (ii) and
(iii) above imply that
\[|q(\x)|=\left|(\bfa.\x)(\bfc.\x)-(\bfb.\x)^2\right|\le
2 ||q||\cdot||\x||^2.\]
Thus they ensure that $q(\x)$ has the expected order of magnitude,
irrespective of the size of $\bfa$, $\bfb$ and $\bfc$.

\begin{proof}
    We begin by observing that in general one has
    \[||M\x||\le||M||\cdot||\x||\]
    for any $3\times 3$ matrix $M$, whence
  \beql{qE}
  |q(\x)|\le \tfrac12 ||\x||\cdot||Q\x||\le\tfrac12||Q||\cdot||\x||^2=
  \tfrac12||q||\cdot||\x||^2.
  \eeq
Taking $\x=\bfe$ we have $|q(\bfe)|\le\tfrac12 ||q||\cdot||\bfe||^2$.
Moreover
\[q(\bfe)=(\bfa.\bfe)(\bfc.\bfe)-(\bfb.\bfe)^2=-(\bfb.\bfe)^2,\]
since $\bfe$ and $\bfc$ are orthogonal.  However
$\bfb.\bfe=||\bfe||^2$, again since $\bfe$ and $\bfc$ are orthogonal.
It follows that
\[||\bfe||^4=|q(\bfe)|\le\tfrac12||q||\cdot||\bfe||^2,\]
and hence that $||\bfe||^2\le\tfrac12 ||q||$.  The first claim of
the lemma then follows.

Alternatively we may take $\x=\bfc$ in (\ref{qE}).  Here we have
\[q(\bfc)=(\bfa.\bfc)(\bfc.\bfc)-(\bfb.\bfc)^2=
\lambda||\bfc||^4-\mu^2||\bfc||^4,\]
whence (\ref{qE}) yields
\[|\lambda-\mu^2|\cdot||\bfc||^4\le\tfrac12||q||\cdot||\bfc||^2.\]
This gives us the second assertion of the lemma.

Thirdly we consider $\bfc Q\bff^T$, where $\bff=\bfd-2\mu\bfe$.
Recalling that $\bfa$ etc. are row vectors, we have
\[\bfc Q\bff^T=\bfc A^T\left(\begin{array}{rrr} 0 & 0 & 1\\
  0 & -2 & 0\\ 1 & 0 & 0\end{array}\right)A\bff^T=
  (\bfc.\bfc)(\bfa.\bff)-2(\bfb.\bfc)(\bfb.\bff)+(\bfa.\bfc)(\bfc.\bff).\]
  However $\bfb.\bfc=\mu||\bfc||^2$, and since $\bff$ is orthogonal to
  $\bfc$ we have $\bfa.\bff=\bfd.\bff$, $\bfb.\bff=\bfe.\bff$ and
$\bfc.\bff=0$, so that
  \[\bfc Q\bff^T=\{(\bfd.\bff)-2\mu (\bfe.\bff)\}||\bfc||^2
  =||\bff||^2||\bfc||^2.\]
  On the other hand
  \[|\bfc Q\bff^T|\le ||\bfc||\cdot||Q\bff^T||\le
  ||\bfc||\cdot||\bff||\cdot||Q||=||\bfc||\cdot||\bff||\cdot||q||.\]
  Thus $||\bff||\le ||q||/||\bfc||$ as in the third claim of the
  lemma.

  The final part is merely a trivial consequence of our initial
  assumption that  $||\bfa||\le ||\bfc||$.
\end{proof}

We are now ready to prove Theorem \ref{t3}.
Suppose we have found a matrix $UA$ with rows $\bfa$, $\bfb$
and $\bfc$, such that $||\bfa||\cdot||\bfc||$ is minimal.
Since $J(UA\x)=J(A\x)=q(\x)$ we may
apply Lemma \ref{ald} to $UA$. Premultiplying $UA$ by
\beql{U0d}
U_2=\left(\begin{array}{rrr} 0 & 0 & 1\\
  0 & 1 & 0\\ 1 & 0 & 0\end{array}\right)\in\Aut(J)
  \eeq
  if necessary we may assume that
  $||\mathbf{a}||\le ||\mathbf{c}||$. Similarly, premultiplying by
  \[\left(\begin{array}{rrr} 1 & 0 & 0\\
  0 & -1 & 0\\ 0 & 0 & 1\end{array}\right)\in\Aut(J)\]
  if necessary, we may assume that $\mu\ge 0$, in the notation of
  Lemma \ref{ald}.

  We begin the proof by observing that it suffices to show that we
have  $||\bfa||\cdot||\bfc||\le 9||q||$.  To see this we note that the
  choice $\x=\bfb$ in (\ref{qE}) yields
  \[\left|(\bfa.\bfb)(\bfc.\bfb)-||\bfb||^4\right|\le\tfrac12
  ||q||\cdot||\bfb||^2,\]
  whence
  \begin{eqnarray*}
    ||\bfb||^4&\le&\tfrac12 ||q||\cdot||\bfb||^2 +|(\bfa.\bfb)(\bfc.\bfb)|\\
    &\le&\{\tfrac12 ||q||+||\bfa||\cdot||\bfc||\}||\bfb||^2\\
    &\le& 10||q||\cdot||\bfb||^2,
    \end{eqnarray*}
given that  $||\bfa||\cdot||\bfc||\le 9||q||$. This gives us the required 
second bound $||\bfb||^2\le 10 ||q||$.

  We also note that if $||\bfc||\le 3\sqrt{||q||}$ then
  \[||\bfa||\cdot||\bfc||\le||\bfc||^2\le 9||q||,\]
  since we are assuming that $||\bfa||\le||\bfc||$. Thus we may suppose
  that
  \beql{s3}
  ||\bfc||\ge3\sqrt{||q||}
  \eeq
  for the remainder of the proof.
  
We now consider $U_1UA$ where
\[U_1=\left(\begin{array}{rrr} 1 & 0 & 0 \\ -1 & 1 & 0\\ 1 & -2 & 1
\end{array}\right).\]
Then $U_1\in \Aut(J)$ and $U_1UA$ has rows
$\bfa$, $-\bfa+\bfb, \bfa-2\bfb+\bfc$. 
Since $UA$ was chosen with $||\bfa||\cdot||\bfc||$ minimal, we conclude that
\[||\bfa||\cdot||\bfa-2\bfb+\bfc||\ge||\bfa||\cdot||\bfc||,\]
and hence that
\[||\bfa-2\bfb+\bfc||\ge||\bfc||.\]
We now substitute $\bfa=\lambda\bfc+\bfd$ and $\bfb=\mu\bfc+\bfe$,
yielding
\[||(\lambda-2\mu+1)\bfc+\bfd-2\bfe||\ge||\bfc||.\]
Thus parts (ii) and (iii) of Lemma \ref{ald} yield
\begin{eqnarray*}
  ||\bfc||&\le& |\lambda-2\mu+1|\cdot||\bfc||+||\bfd-2\bfe||\\
  &\le&|\mu^2-2\mu+1|\cdot||\bfc||+\frac{||q||}{2||\bfc||}+
\frac{||q||}{||\bfc||}+2|\mu-1|\cdot||\bfe||.
\end{eqnarray*}
However parts (ii) and (iv) of the lemma, along with our assumption (\ref{s3}),
show that
\[\mu^2\le |\lambda|+||q||/(2||\bfc||^2)\le \frac{19}{18}.\]
Since we are assuming that $\mu\ge 0$ we conclude that
$0\le\mu\le 37/36$.  In particular $|1-\mu|\le 1$ so that (\ref{s3}) yields
\begin{eqnarray*}
  ||\bfc||&\le& (1-\mu)^2||\bfc||+\frac{3||q||}{2||\bfc||}+2||\bfe||\\
  &\le& (1-\mu)^2||\bfc||+\frac{\sqrt{||q||}}{2}+\sqrt{2||q||}\\
  &\le& (1-\mu)^2||\bfc||+2\sqrt{||q||},
\end{eqnarray*}
by part (i) of Lemma \ref{ald}.  It now follows that
\beql{pp}
\mu(2-\mu)\le\frac{2\sqrt{||q||}}{||\bfc||},
\eeq
and since $0\le\mu\le{37/36}$ we deduce that
\[\frac{35}{36}\mu\le\frac{2\sqrt{||q||}}{||\bfc||},\]
whence 
  \[0\le\mu\le \frac{72}{35}\frac{\sqrt{||q||}}{||\bfc||}.\]
  
  From part (ii) of Lemma \ref{ald} we now have
  \[|\lambda|\le \frac{||q||}{2||\bfc||^2}+
  \left(\frac{72}{35}\right)^2\frac{||q||}{||\bfc||^2}\le
  5\frac{||q||}{||\bfc||^2}.\]
 Moreover, parts (i) and (iii) yield
  \[||\bfd||\le ||q||/||\bfc||+\sqrt{2}|\mu|\sqrt{||q||}
  \le \left(1+\frac{72\sqrt{2}}{35}\right)\frac{||q||}{||\bfc||}\le
  4\frac{||q||}{||\bfc||}.\]
  We therefore conclude that
  \[||\bfa||\le|\lambda|\cdot||\bfc||+||\bfd||
  \le 9\frac{||q||}{||\bfc||},\]
  which suffices for the theorem.

\section{Deduction of Theorem \ref{t2}}\label{pt2}
Theorem \ref{t2} will follow from Theorem \ref{t3}. We have
\[M_k^TQM_k=D_k\left(\begin{array}{rrr} 0 & 0 & 1\\ 0 & -2 & 0\\
  1 & 0 & 0\end{array}\right).\]
  In general one has $M\Adj(M)=\det(M)I$, so that
  \[\det(M_k)^2 Q=D_k A^T\left(\begin{array}{rrr} 0 & 0 & 1\\ 0 & -2 & 0\\
    1 & 0 & 0\end{array}\right) A,\]
    with $A=\Adj(M_k)$.  It follows that
    \[J(A\x)=\det(M_k)^2D_k^{-1}q(\x),\]
    where
    \beql{Inv}
    AM_k=\det(M_k)I.
    \eeq

    We may now apply Theorem \ref{t3}, which provides a matrix
$U\in\Aut(J)$ such that the rows $\bfa,\bfb,\bfc$ of $UA$ satisfy
    \beql{9i}
    ||\bfa||\cdot||\bfc||\le 9\det(M_k)^2D_k^{-1}||q||
    \eeq
    and
    \beql{10i}
    ||\bfb||^2\le 10\det(M_k)^2D_k^{-1}||q||.
    \eeq
    If the columns of $M_kU^{-1}$  are $\bfc_1,\bfc_2,\bfc_3$, and $U_2$
    is given by (\ref{U0d}), then the columns of 
    $M_kU^{-1}U_2^{-1}$ are $\bfc_3, -\bfc_2,\bfc_1$, while the rows
    of $U_2UA$ are $\bfc,-\bfb,\bfa$. Thus we are free to replace $U$
    by $U_2U$ if we wish.
    We may therefore suppose without loss of generality that the columns
    of $M_kU^{-1}$ have $||\bfc_1||\le||\bfc_3||$. We may also
replace $U$ by $-U$, which will not affect the properties
(\ref{9i}) and (\ref{10i}) or the lengths $||\bfc_1||$ and
$||\bfc_3||$. Thus we may also suppose without loss of generality
that $\det(U)=+1$. 

Having suitably modified $U$ we still have 
\[J(UA\x)=\det(M_k)^2D_k^{-1}q(\x),\]
so that
\[(UA)^T\left(\begin{array}{rrr} 0 & 0 & 1\\ 0 & -2 & 0\\
  1 & 0 & 0\end{array}\right)UA=\det(M_k)^2D_k^{-1}Q.\]
  We now claim that we may replace $M_k$ by $N_k=M_kU^{-1}$ in 
  Theorem~ \ref{t1}. Part (i) of the theorem obviously remains true.
  Since $\det(N_k)=\det(M_k)$ the second and sixth  assertions of Theorem
\ref{t1} are immediate.  Moreover
\[q(N_k\x)=q\big(M_k(U^{-1}\x)\big)=D_kJ(U^{-1}\x)=D_kJ(\x),\]
giving us the fourth assertion, and also the fifth since the value of $D_k$ is the same for $N_k$ as it was for $M_k$. Finally, $U^{-1}(\Z^3)=\Z^3$ since
$\det(U)=1$, whence $N_k(\Z^3)=M_k(\Z^3)$. This suffices for the
third assertion of the theorem.

We proceed to consider the rows of $N_k^{-1}$. Since $N_k=M_kU^{-1}$, we have
\[\Adj(N_k)=\Adj(U^{-1})\Adj(M_k)=UA=B.\]
Thus
\[N_k^{-1}=\det(N_k)^{-1}\Adj(N_k)=\det(N_k)^{-1}B.\]
Since $\det(N_k)=\det(M_k)$ the first pair of inequalities in
Theorem \ref{t2} now follow from (\ref{9i}) and (\ref{10i}).

To handle the columns $\b{c}_i$ of $N_k$ we begin with the observation that
\[N_k=\det(N_k)\Adj(N_k^{-1}).\]
If the rows of $N_k^{-1}$ are $\b{r}_1=\b{u}$, $\b{r}_2=\b{v}$ and
$\b{r}_3=\b{w}$, then the first column of $\Adj(N_k^{-1})$ will be
\[\big(v_2w_3-v_3w_2,v_3w_1-v_1w_3,v_1w_2-v_2w_1\big)^T,\]
and hence will have Euclidean length at most
$||\mathbf{v}||\cdot||\mathbf{w}||$.  It follows that
\beql{c1E}
||\b{c}_1||\le \det(N_k)||\b{r}_2||\cdot||\b{r}_3||,
\eeq
and similarly that
\beql{c23E}
||\b{c}_2||\le \det(N_k)||\b{r}_1||\cdot||\b{r}_3||\;\;\;\mbox{and}
\;\;\; ||\b{c}_3||\le \det(N_k)||\b{r}_1||\cdot||\b{r}_2||.
\eeq
Thus (\ref{c1E}) and the second part of (\ref{c23E}) yield
\[||\b{c}_1||\cdot||\b{c}_3||\le
\det(N_k)^2||\b{r}_1||\cdot||\b{r}_2||^2||\b{r}_3||,\]
so that the first inequality of (\ref{t213}) follows from
(\ref{t2crE}).  The second part of (\ref{t213}) is then a consequence
of the fact that $\det(N_k)\mid D_k$, as noted in Theorem \ref{t1}.
To establish (\ref{t22}) we merely combine the first part of (\ref{c23E}) with
(\ref{t2crE}),  and again use the fact that $\det(N_k)\mid D_k$. This
completes the proof of Theorem \ref{t2}.

\section{Proof of Theorem \ref{t5}}

The matrix (\ref{QD}) has three real eigenvalues, whose product is
 $\det(Q)=2\Delta$. For any vector $\x$ we have 
 $||Q\x||\le  ||Q||\cdot ||\x||$, so that if $\lambda$ is an eigenvalue we 
 must have $|\lambda|\le||Q||$. Since $||q||$ is defined to be $||Q||$ 
 we conclude that $2\Delta\le||q||^3$, giving us the required 
 bound $\rho(q)\ge 2$.
 
  If $\z_1,\z_2\in\Z^3$ are linearly independent zeros of $q$ from the
  same class class $\cl{C}_k$, then $q(\z_1+\z_2)$ cannot vanish, since
  a non-singular conic cannot have three collinear zeros.  However
  $\z_1+\z_2$ will be in $M_k(\Z^3)$ so that we
  must have $D_k\mid q(\z_1+\z_2)$ by (\ref{fun}). It follows that
  $q(\z_1+\z_2)-q(\z_1)-q(\z_2)$ is a non-zero multiple of $D_k$.
  Recalling the definition (\ref{QD}) of the matrix $Q$ of $q$ we see that
  $q(\z_1+\z_2)-q(\z_1)-q(\z_2)=\z_1^TQ\z_2$. We therefore find that
  \[D_k\le|q(\z_1+\z_2)-q(\z_1)-q(\z_2)|=|\z_1^TQ\z_2|\le
  ||\z_1||\cdot||Q||\cdot||\z_2||,\]
  and hence that $||\z_1||\cdot||\z_2||\ge D_k/||q||$.  This gives us
  the second assertion of the theorem.

  Next, if $\b{c}_1$ is not a scalar multiple of $\z_1$ we will have
  \[||\b{c}_1||\ge ||\z_2||\ge||\z_1||. \]
  This would lead to the  inequalities
  \[D_k/||q||\le ||\z_1||\cdot||\z_2||\le||\b{c}_1||^2
< \rho^{-1}\det(M_k)^2D_k^{-2}||q||^2.\]
  We then have a contradiction, by virtue of (\ref{MD}). This
  establishes the third claim of the theorem.

  Finally, since $||\b{c}_3||\ge||\z_2||$ we have
  \begin{eqnarray*}
   ||\b{c}_1||&\le& 90\frac{\det(M_k)^2D_k^{-2}||q||^2}{||\b{c}_3||}\\
&  \le& 90\frac{\det(M_k)^2D_k^{-2}||q||^2}{||\z_2||}\\
&  \le& 90\frac{\det(M_k)^2D_k^{-2}||q||^2}{D_k/||q||}||\z_1||\\
   &  =&90\frac{||q||^3}{\Delta}||\z_1||,
   \end{eqnarray*}
  by (\ref{t213}), (\ref{z1z2}) and (\ref{MD}). Similarly, since
  $||\b{c}_1||\ge||\z_1||$ we have  
  \begin{eqnarray*}
   ||\b{c}_3||&\le& 90\frac{\det(M_k)^2D_k^{-2}||q||^2}{||\b{c}_1||}\\
&  \le& 90\frac{\det(M_k)^2D_k^{-2}||q||^2}{||\z_1||}\\
&  \le& 90\frac{\det(M_k)^2D_k^{-2}||q||^2}{D_k/||q||}||\z_2||\\
   &  =&90\frac{||q||^3}{\Delta}||\z_2||.
   \end{eqnarray*}
 This completes our proof of Theorem \ref{t5}.

\section{Preliminaries for the proof of Theorem \ref{t4}}

If $M=M_k$ we see from Theorem \ref{t1} that
\[N(B;\mathcal{C}_k)=\frac{1}{2}
\sum_{\substack{\u\in\Z^2\\ M\u^2\mathrm{primitive}}}\left\{w(B^{-1}M\u^2)
+w(-B^{-1}M\u^2)\right\},\]
where we write
\[M\u^2=M\left(\begin{array}{c} u_1^2\\ u_1u_2\\ u_2^2\end{array}\right)
  \;\;\; \mbox{when}\;\; \u=(u_1,u_2),\]
  for notational convenience.  It is thus also convenient to set
  $w_+(\x)=w(\x)+w(-\x)$ so that $w_+$ is an even function, supported
  on the set $||\x||\le 1$. We then have
  \[N(B;\mathcal{C}_k)=\frac{1}{2}
  \sum_{\substack{\u\in\Z^2\\ M\u^2\mathrm{primitive}}}w_+(B^{-1}M\u^2).\]

We begin by considering the condition that $M\u^2$ should be
primitive. Our goal is the following result.
\begin{lemma}\label{lprim}
    The set
  of primes may be partitioned into sets
  $\mathcal{P}_0,\mathcal{P}_1, \mathcal{P}_2$ with the following properties.
  Firstly, if $p\nmid\Delta$ then $p\in\mathcal{P}_0$.  Secondly, if
  $p\in\mathcal{P}_0$ then $p\mid M\u^2$ if and only if $p\mid\u$.
  Thirdly, if
  $p\in\mathcal{P}_n$ for $n=1$ or 2 then there are distinct lattices
  $\Lambda_i(p)\subseteq\Z^2$ for $1\le i\le n$ having determinant
  $p$, and such that $p\mid M\u^2$ if and only if
  $\u$ lies in one of the lattices $\Lambda_i(p)$. Finally, if
  $p||\Delta$ then $p\in\mathcal{P}_1$.
\end{lemma}

\begin{proof}
If $p$ is a prime not dividing $\det(M)$ (and in particular
for any prime not dividing $\Delta$) the matrix $M$ will be
invertible modulo $p$, so that $p\mid M\u^2$ if and only if
$p\mid\u^2$.  In this case the condition that $p\mid M\u^2$ is
equivalent to $p\mid\u$, and $p$ will be in $\mathcal{P}_0$.
On the other hand, if $p\mid \det(M)$ then $M$
is singular modulo $p$.  It cannot vanish modulo $p$, since the class
$\cl{C}_k$ corresponding to $M=M_k$ is assumed to be non-empty, whence
$M$ has rank 1 or 2 modulo $p$. Suppose firstly that
$M$ has rank 1 over $\mathbb{F}_p$, with a non-zero row $(A,B,C)$
say.  If the quadratic form $Au^2+Buv+Cv^2$ is irreducible modulo $p$
then $p\mid M\u^2$ implies $p\mid\u$. In this case $p$ will be in
$\mathcal{P}_0$. If the form $Au^2+Buv+Cv^2$ splits into distinct factors
as $L_1(u,v)L_2(u,v)$ then $p\mid M\u^2$ if and only if $\u$ lies in
one or both of the lattices $\Lambda_1,\Lambda_2$ given by 
$p\mid L_i(u,v)$. In this case $p$ will be in $\mathcal{P}_2$.
On the other hand, if $Au^2+Buv+Cv^2$ has a repeated factor
$L(u,v)^2$, then one has $p\mid M\u^2$ if and only if $\u$ lies in the
lattice $\Lambda$ given by $p\mid L(u,v)$, so that
$p\in\mathcal{P}_1$. A similar analysis
applies when $M$ has rank 2 over $\mathbb{F}_p$, showing in this case
that the condition $p\mid M\u^2$ is either equivalent to
$\u\in\Lambda$ for some lattice $\Lambda\subset\Z^2$ of determinant
$p$, or is equivalent to $p\mid\u$.

Finally, suppose that $p||\Delta$. Then $p||\det(M)$, by part (iii) of
Theorem \ref{t1}, so that $M$ has rank 2 over $\mathbb{F}_p$. Using
row operations one sees that there is a matrix $U\in\mathrm{GL}_3(\Z_p)$
such that $UM=R$ takes one of the forms
\[R_1=\left(\begin{array}{ccc}1 & 0 & a\\ 0 & 1 & b\\ 0&0&
  p\end{array}\right)\;\;\mbox{or}\;\;
    R_2=\left(\begin{array}{ccc}1 & a & 0\\ 0 & p & 0\\ 0&0&
      1\end{array}\right)\;\;\mbox{or}\;\;
        R_3=\left(\begin{array}{ccc}p & 0 & 0\\ 0 & 1 & 0\\ 0&0&
  1\end{array}\right).\]
Then if $M=M_k$ and $D=D_k$ we have $p||D$ by (\ref{MD}), and 
\begin{eqnarray*}
J(\Adj(R)\x)&=&\det(M)^2J(R^{-1}\x)\\
&=&\det(M)^2D^{-1}q(MR^{-1}\x)\\
&=&\det(M)^2D^{-1}q(U^{-1}\x).
\end{eqnarray*}
Since $U$ is invertible modulo $p$ we conclude that $J(\Adj(R)\x)$
vanishes modulo $p$. When $R=R_1$ we have
\[\Adj(R_1)=\left(\begin{array}{ccc}p & 0 & -a\\ 0 & p & -b\\ 0&0&
  1\end{array}\right)\]
  so that
  \[J(\Adj(R_1)\x)=(px_1-ax_3)x_3-(px_2-bx_3)^2\equiv -(a+b^2)x_3^2\modd{p}.\]
In this case we conclude that $p\mid a+b^2$. Since $U$ is
invertible modulo $p$ the condition $p\mid M\u^2$ is equivalent to
$p\mid R\u^2$, and for $R=R_1$ this becomes
\[u_1^2+au_2^2\equiv u_1u_2+bu_2^2\equiv 0\modd{p}.\]
Since $a\equiv -b^2\modd{p}$ this holds precisely when
$p\mid u_1+bu_2$. Thus for $R=R_1$ there is a single lattice condition.

For $R=R_2$ we calculate that
\[J(\Adj(R_2)\x)=(px_1-ax_2)px_3-x_2^2,\]
which cannot
vanish identically modulo $p$.  This case is therefore forbidden. When
$R=R_3$ we see that $p\mid R\u^2$ if and only if $p\mid u_2$, which
again gives us a single lattice condition. Thus whenever $p||\Delta$
the condition $p\mid M\u^2$ gives us a single lattice condition with
determinant $p$.
\end{proof}

Lemma \ref{lprim} allows us to handle the primitiveness condition in
the definition of the sum $N(B;\mathcal{C}_k)$ as follows.
\begin{lemma}\label{lSlat}
  Suppose that $w(\x)$ is supported on the disc $||\x||\le 1$.
Then there is a square-free divisor $\Delta_1\Delta_2$ of $\Delta$, 
and lattices $\Lambda^{(1)},\ldots,\Lambda^{(J)}$, where
\[J=2^{\omega(\Delta_1)}3^{\omega(\Delta_2)},\]
with the following properties.  Firstly if $p||\Delta$ then
$p\mid\Delta_1$; secondly the determinant
$\dL(\Lambda^{(j)})$ divides
$\Delta_1\Delta_2^2$ for every index $j$; thirdly
\[N(B;\mathcal{C}_k)=\frac{1}{2}
\sum_{j=1}^J\lambda\left(\dL(\Lambda^{(j)})\right)
\sum_{\substack{d=1\\ (d,\Delta_1\Delta_2)=1}}^{\infty}\mu(d)
\sum_{\u\in\Lambda^{(j)}-\{\mathbf{0}\}}w_+(d^2B^{-1}M\u^2),\]
where $\lambda(n)=(-1)^{\Omega(n)}$ is the Liouville function; and fourthly
\[\sum_{j=1}^J\frac{\lambda\left(\dL(\Lambda^{(j)})\right)}{\dL(\Lambda^{(j)})}
\sum_{\substack{d=1\\ (d,\Delta_1\Delta_2)=1}}^{\infty}\frac{\mu(d)}{d^2}=
\frac{6}{\pi^2}\prod_{p\mid \Delta_1}\frac{1}{1+p^{-1}}
\prod_{p\mid \Delta_2}\frac{1-p^{-1}}{1+p^{-1}}.\]
\end{lemma}

\begin{proof}
For the proof we use the notation $\one(A)$ for the characteristic
function for the property $A$.  We begin by observing that
\[\one(p\nmid M\u^2)=1-\one(p\mid\u),\;\;\; p\in\mathcal{P}_0,\]
\[\one(p\nmid M\u^2)=1-\one(\u\in\Lambda_1(p)),\;\;\; p\in\mathcal{P}_1,\]
and finally,
\[\one(p\nmid M\u^2)=1-\one(\u\in\Lambda_1(p))-\one(\u\in\Lambda_2(p))
+\one(\u\in\Lambda_1(p)\cap\Lambda_2(p)),\]
when $p\in\mathcal{P}_2$.
We now take $\Delta_1$ to be the product of the primes in $\mathcal{P}_1$
and $\Delta_2$ to be the product of the primes in $\mathcal{P}_2$, so that
$\Delta_1\Delta_2\mid \Delta$. Let $\Lambda^{(j)}$ (for $1\le j\le J$) run over
all lattices formed by the intersection of none, some, or all, of the
lattices $\Lambda_i(p)$ (for $i=1$ or 2 and
$p\in\mathcal{P}_1\cup\mathcal{P}_2$).  Then
$J=2^{\omega(\Delta_1)}3^{\omega(\Delta_2)}$, and
\[\one\left((\Delta_1\Delta_2,M\u^2)=1\right)=
\sum_{j=1}^J\lambda\left(\dL(\Lambda^{(j)})\right)\one(\u\in\Lambda^{(j)}),\]
since $\Lambda_1(p)\cap\Lambda_2(p)$
will have determinant $p^2$ when $p\mid \Delta_2$.
The conditions $p\nmid M\u^2$ for primes $p\in\mathcal{P}_0$ are
produced by
\[\sum_{\substack{d\mid\u\\ (d,\Delta_1\Delta_2)=1}}\mu(d),\]
so that
\begin{eqnarray*}
  N(B;\mathcal{C}_k)&=&\frac{1}{2}
  \sum_{\substack{\u\in\Z^2\\ M\u^2\mathrm{primitive}}}
  w_+\left(\frac{M}{B}\u^2\right)\\
&=&\frac{1}{2}
  \sum_{\substack{\u\in\Z^2-\{\mathbf{0}\}\\ M\u^2\mathrm{primitive}}}
  w_+\left(\frac{M}{B}\u^2\right)\\
&=&\frac{1}{2} \sum_{j=1}^J\lambda\left(\dL(\Lambda^{(j)})\right)
  \sum_{\substack{d=1\\ (d,\Delta_1\Delta_2)=1}}^{\infty}\mu(d)
  \sum_{\substack{\u\in\Lambda^{(j)}-\{\mathbf{0}\}\\ d\mid\u}}w_+\left(\frac{M}{B}\u^2\right)\\
  &=&\frac{1}{2}
  \sum_{j=1}^J\lambda\left(\dL(\Lambda^{(j)})\right)
\sum_{\substack{d=1\\ (d,\Delta_1\Delta_2)=1}}^{\infty}\mu(d)
\sum_{\u\in\Lambda^{(j)}-\{\mathbf{0}\}}w_+\left(d^2\frac{M}{B}\u^2\right)
\end{eqnarray*}
as required. Here we should note that the $d$-summation is finite for
all relevant $\u$.

The final part is clear, by multiplicativity.
\end{proof}

In light of Lemma \ref{lSlat} our focus moves to sums of the shape
\[S(\Lambda,B,M_k)=\sum_{\x\in\Lambda-\{\mathbf{0}\}}w_+(B^{-1}M_k\x^2),\]
where $\Lambda$ is an integer lattice, $w_+$ is an even weight supported in
the disc $||\x||\le 1$, and $M_k$ is an integer matrix of
the shape described in Theorems \ref{t1} and \ref{t2}.
We first need to understand the range of summation in $S(\Lambda,B,M_k)$.

\begin{lemma}\label{IE}
Let $M$ be one of the matrices $M_k$, as described in Theorem
\ref{t2}. Let $X_1=\sqrt{B||\b{r}_1||}$ and $X_2=\sqrt{B||\b{r}_3||}$,
so that
\[X_1X_2\le 3BD_k^{-1/2}||q||^{1/2}.\]
Then if $w_+(B^{-1}M\x^2)\not=0$ with $\x\in\R^2$
we have both $|x_1|\le X_1$ and $|x_2|\le X_2$. Moreover if
$w_+(B^{-1}M\x^2)\not=0$ with $\x\in\Z^2$ we have $B\ge 1$.
\end{lemma}

\begin{proof}
We set $B^{-1}M\x^2=\y$, so that $||\y||\le 1$ if $w_+(B^{-1}M\x^2)\not=0$.  
If $M^{-1}$ has rows $\b{r}_1,\b{r}_2,\b{r}_3$, as in Theorem
\ref{t2}, then
\[|(M^{-1}\y)_1|=|\b{r}_1^T\y|\le ||\b{r}_1||\cdot||\y||\le
||\b{r}_1||,\]
so that $x_1^2\le B||\b{r}_1||$.  Similarly $x_2^2\le B||\b{r}_3||$,
and the first result follows. If $\x\in\Z^2-\{\mathbf{0}\}$ with
$w_+(B^{-1}M\x^2)\not=0$ we have $B^{-1}||M\x^2||\le 1$. Since $\x^2$
does not vanish we see that $M\x^2$ must be a non-zero integer vector,
since $M$ is nonsingular.  It follows that $||M\x^2||\ge 1$, whence
$B\ge 1$ as claimed.
\end{proof}

We now give a crude bound for $S(\Lambda,B,M_k)$.
\begin{lemma}\label{CB}
  We have $S(\Lambda,B,M_k)=0$ if $B<1$, and otherwise
  \[S(\Lambda,B,M_k)\ll_wD_k^{-1/2}||q||^{1/2}\left\{\frac{B}{\dL(\Lambda)}+
  B^{1/2}||\b{c}_3||^{1/2}\right\}.\]
\end{lemma}
\begin{proof}
  The first claim is obvious, given Lemma \ref{IE}.  Generally
\[S(\Lambda,B,M_k)\ll_w \card\{\x\in\Lambda:\,|x_1|\le X_1,\,
    |x_2|\le X_2\}.\]
    If we set
    \[\Lambda_0=\{(x_1/X_1,x_2/X_2):\,(x_1,x_2)\in\Lambda\},\]
    then $\dL(\Lambda_0)=\dL(\Lambda)/X_1X_2$, and
    \[S(\Lambda,B,M_k)\ll_w \card\{\y\in\Lambda_0:\,||\y||_{\infty}\le 1\}.\]
    Thus
    \[S(\Lambda,B,M_k)\ll_w \dL(\Lambda_0)^{-1}+\lambda_1^{-1}+1
    \ll X_1X_2\dL(\Lambda)^{-1}+\lambda_1^{-1}+1,\]
    where $\lambda_1$ is the length of the shortest non-zero vector in
    $\Lambda_0$. However one has $||\x||\ge 1$ for every non-zero
    vector in $\Lambda$, and hence
    \[\lambda_1\ge\max(X_1,X_2)^{-1}.\]
    We therefore obtain the bound
    \[S(\Lambda,B,M_k)\ll_w \frac{X_1X_2}{\dL(\Lambda)}+\max(X_1,X_2)+1.\]
    If $\max(X_1,X_2)\le \tfrac12$ and $\x$ is an integer vector for which
    $w_+(B^{-1}M\x^2)$ is non-zero, then we must have $\x=\mathbf{0}$,
    which is excluded
    from the sum $S(\Lambda,B,M_k)$.  Thus $S(\Lambda,B,M_k)=0$ when
$\max(X_1,X_2)\le\tfrac12$. It therefore follows that
\begin{eqnarray*}
S(\Lambda,B,M_k)&\ll_w&\frac{X_1X_2}{\dL(\Lambda)}+\max(X_1,X_2)\\
&\ll_w& D_k^{-1/2}||q||^{1/2}\frac{B}{\dL(\Lambda)}+
  B^{1/2}\max\{||\b{r}_1||\,,\,||\b{r}_3||\}^{1/2}.
  \end{eqnarray*}
    We now claim that
  \beql{cln}
  ||\b{r}_1||\ll \frac{||q||}{D_k}||\b{c}_3||
  \;\;\;\mbox{and}\;\;\; ||\b{r}_3||\ll
  \frac{||q||}{D_k}||\b{c}_1||.
\eeq
Clearly this will suffice for the lemma, since we have chosen $M_k$ so
that $||\b{c}_1||\le||\b{c}_3||$.

To prove (\ref{cln}) we begin with the observation that the scalar product
$\b{r}_1.\b{c}_1$ takes the value $1$, since $M^{-1}M=I$. Similarly we have
$\b{r}_3.\b{c}_3=1$. It follows that
\beql{alp}
||\b{r}_1||\cdot||\b{c}_1||\ge 1\;\;\;\mbox{and}\;\;\;
  ||\b{r}_3||\cdot||\b{c}_3||\ge 1.
  \eeq
Thus
\[ ||\b{r}_1||\le||\b{r}_1||\cdot||\b{r}_3||\cdot||\b{c}_3||,\]
so that the first part of (\ref{cln}) follows from (\ref{t2crE}).
The second part may be proved entirely analogously.
\end{proof}

\section{Theorem \ref{t4} --- The leading term}

To estimate $S(\Lambda,B,M_k)$ we will use the
following form of the Poisson summation formula. 

\begin{lemma}\label{poisson}
  Let $N\in$GL$_2(\R)$, so that $\Lambda=N(\Z^2)$ is a two-dimensional
  lattice. Then
  \[S(\Lambda,B,M_k)+w_+(\mathbf{0})
  =\dL(\Lambda)^{-1}\sum_{\bfa\in\Z^2}I(\bfa,M,N),\]
with
  \[I(\bfa,M,N)=
\int_{\R^2}w_+(B^{-1}M\x^2)e(-\bfa^TN^{-T}\x)dx_1dx_2.\]
\end{lemma}
\begin{proof}
  Writing $\varpi(\x)=w_+(B^{-1}M\x^2)$, the Poisson summation formula yields
  \[S(\Lambda,B,M_k)+w_+(\mathbf{0})
  =\sum_{\y\in\Z^2}\varpi(N\y)=\sum_{\bfa\in\Z^2}
\int_{R^2}\varpi(N\z)e(-\bfa^T\z)dz_1dz_2.\]
If we substitute $\x=N\z$ we have $\bfa^T\z=\bfa^TN^{-T}\x$, and the
result follows since $|\det(N)|=\dL(\Lambda)$.
\end{proof}

The main term in Theorem \ref{t4}
will come from the integral with $\bfa=\mathbf{0}$.
\begin{lemma}\label{MT}
  Define $K_T(t)$ as in (\ref{KTD}).
Then if $q(M\x)=DJ(\x)$ as in Theorem \ref{t1} we have
\[\int_{\R^3}w(\y)K_T(q(\y))dy_1dy_2dy_3\to\frac{D^{1/2}}{\Delta^{1/2}}
\int_{\R^2}w_+(M\x^2)dx_1dx_2,\]
as $T\to\infty$, and hence
\[\int_{\R^2}w_+(B^{-1}M\x^2)dx_1dx_2=\sigma_{\infty}(q;w)
\frac{\Delta^{1/2}}{D^{1/2}}B,\]
where $\sigma_{\infty}(q;w)$ is given by (\ref{sid}).
\end{lemma}

\begin{proof}
Since
\[\int_{\R^3}w(\y)K_T(q(\y))dy_1dy_2dy_3=
\int_{\R^3}w(-\y)K_T(q(\y))dy_1dy_2dy_3\]
we have
\[\sigma_{\infty}(q;w)=\frac{1}{2}\lim_{T\to\infty}\int_{\R^3}
w_+(\y)K_T(q(\y))dy_1dy_2dy_3.\]
  Then, writing $\tilde{w}(\x)=w_+(D^{-1/2}M\x)$ we have
  \begin{eqnarray}\label{et2}
    \lefteqn{\int_{\R^3}w_+(\y)K_T(q(\y))dy_1dy_2dy_3}\nonumber\\
    &=&
    \det(M)\int_{\R^3}w_+(M\z)K_T(q(M\z))dz_1dz_2dz_3\nonumber\\
    &=&\det(M)\int_{\R^3}w_+(M\z)K_T(DJ(\z))dz_1dz_2dz_3\nonumber\\
    &=&\frac{\det(M)}{D^{3/2}}
    \int_{\R^3}\tilde{w}(\x)K_T(J(\x))dx_1dx_2dx_3\nonumber\\
    &=&\Delta^{-1/2}\int_{\R^3}\tilde{w}(\x)K_T(J(\x))dx_1dx_2dx_3.
  \end{eqnarray}
  The function $\tilde{w}$ will be even, so that the above becomes
\[  \frac{2}{\Delta^{1/2}}\int_0^{\infty}
  \int_{\R^2}\tilde{w}(\x)K_T(J(\x))dx_1dx_2dx_3.\]
 Since $w_+(\y)$ is supported on the set $||\y||\le 1$ we see that
 $\tilde{w}(\z)$ is supported on a subset of $[-C,C]^3$ for some
 $C=C(M,D)>0$.  If we write $x_0=x_2^2/x_3$ and $x_1=x_0+u$, then
 both $\tilde{w}(\x)$ and $\tilde{w}(x_0,x_2,x_3)$ vanish unless
$|x_2|\le C$ and $x_3\le C$. If $\phi$ is the function
\[\phi(u,x_2,x_3)=\tilde{w}(x_0+u,x_2,x_3)-\tilde{w}(x_0,x_2,x_3)\]
it follows that
\begin{eqnarray}\label{cc}
  \lefteqn{\int_0^{\infty}\int_{\R^2}\{\tilde{w}(\x)-\tilde{w}(x_0,x_2,x_3)\}
    K_T(J(\x))dx_1dx_2dx_3}\nonumber\\
&=&\int_0^C\int_{-C}^C\int_{-\infty}^{\infty}\phi(u,x_2,x_3)
K_T(x_3u)dudx_2dx_3\nonumber\\
&\ll& T\int_0^C\int_{-C}^C\int_{|u|\le 1/Tx_3}
|\phi(u,x_2,x_3)|dudx_2dx_3.
\end{eqnarray}
We may assume that $\tilde{w}(x_0+u,x_2,x_3)$ and
$\tilde{w}(x_0,x_2,x_3)$ do not both vanish, and hence that
either $|x_0|\le C$ or $|x_0+u|\le C$, (or both). In this case
\beql{et1}
 \phi(u,x_2,x_3)\ll_{w,M}\min(1,|u|).
 \eeq
 We proceed to consider separately the ranges $|x_0|\ge 2C$ and
 $|x_0|\le 2C$.
 When $|x_0|\ge 2C$ the bound $|x_0+u|\le C$ implies that
 $|u|\ge x_0/2$, and since $|u|\le 1/Tx_3$ we conclude that
 $x_2^2\le 2/T$. The bound $|x_0|\ge 2C$ then shows that
 $x_3\le x_2^2/2C\le 1/CT$. Moreover $u$ is restricted to a range
 $|x_0+u|\le C$ of length $O_M(1)$.  The estimate in (\ref{et1}) is
 $O_{w,M}(1)$, so that
the corresponding contribution to (\ref{cc}) is
\[\ll_{w,m}T\int_0^{1/CT}
\int_{x_2^2\le  2/T}dx_2\,dx_3\ll_{w,M}T^{-1/2}.\]
On the other hand, when $|x_0|\le 2C$ we have $x_2^2\le 2Cx_3$ and
\[\int_{|u|\le 1/Tx_3}\min(1,|u|)du\ll
(Tx_3)^{-1}\min(1\,,\,1/Tx_3),\]
so that the corresponding contribution to (\ref{cc}) is
\[\ll_{w,M}\int_0^C x_3^{-1/2}\min(1\,,\,1/Tx_3)dx_3\ll_{w,M}
T^{-1/2}.\]
We therefore conclude that
  \begin{eqnarray*}
    \lefteqn{\int_{\R^3}w_+(\y)K_T(q(\y))dy_1dy_2dy_3}\\
    &=&\frac{2}{\Delta^{1/2}}\int_0^{\infty}\int_{\R^2}
    \tilde{w}(x_0,x_2,x_3)K_T(J(\x))dx_1dx_2dx_3+O_{w,M}(T^{-1/2}).
  \end{eqnarray*}
  The function $\tilde{w}(x_0,x_2,x_3)$ is independent of $x_1$ and
 \[ \int_{\R}K_T(x_1x_3-x_2^2)dx_1=x_3^{-1}.\]
Moreover with the substitutions $x_3=u_2^2$ and $x_2=u_1u_2$ we have
\begin{eqnarray*}
 \int_0^{\infty}\int_{\R}\tilde{w}(x_2^2/x_3,x_2,x_3)x_3^{-1}dx_2dx_3
&=&2\int_0^{\infty}\int_{\R}\tilde{w}(u_1^2,u_1u_2,u_2^2)du_1\,du_2\\
&=&\int_{\R^2}\tilde{w}(u_1^2,u_1u_2,u_2^2)du_1\,du_2\\
&=&D^{1/2}\int_{\R^2}w_+(M\u^2)du_1\,du_2,
\end{eqnarray*}
and the lemma follows.
\end{proof}

Our next result tells us about the size of
$\sigma_{\infty}(q;w)$. Recall that $q$ is isotropic, and hence
indefinite, with $\Delta>0$, so that the matrix $Q$ of $q$ has one
positive eigenvalue $\lambda$ say, and two negative ones $-\mu$ and
$-\nu$ say.  We may assume that $\mu\ge\nu(>0)$. With this notation
we have $\Delta=\tfrac12\lambda\mu\nu$. We remind the reader of the
notation $f\asymp g$, meaning that both $f\ll g$ and $g\ll f$ hold.
in our context the two implied constants will be absolute. Thus we
will have
\[\max(\lambda,\mu,\nu)\asymp||q||,\]
for example.
\begin{lemma}\label{cqE}
  If $\sup|w(\x)|\le 1$ we have
  \[|\sigma_{\infty}(q;w)|\le 2e^{4/3} \sigma_{\infty}(q;w_0),\]
  where the weight $w_0$ is given by (\ref{w0d}).  Moreover
  \[\sigma_{\infty}(q;w_0)\asymp
  \frac{\min(\lambda,\mu,\nu)^{1/2}}{\Delta^{1/2}}\log(2\mu/\nu)\]
    when $\lambda\ge\mu \ge\nu$,
    \[\sigma_{\infty}(q;w_0)\asymp
  \frac{\min(\lambda,\mu,\nu)^{1/2}}{\Delta^{1/2}}\log(2\lambda/\nu)\]
    when $\mu\ge\lambda \ge\nu$, and
    \[\sigma_{\infty}(q;w_0)\asymp
  \frac{\min(\lambda,\mu,\nu)^{1/2}}{\Delta^{1/2}}\]
  when $\mu\ge\nu \ge\lambda$.

  Thus
  \[\frac{1}{||q||}\ll \sigma_{\infty}(q;w_0)\ll \frac{\rho^{1/4}}{||q||}\]
  in every case.
\end{lemma}
\begin{proof}
  According to Lemma \ref{MT} we have
  \begin{eqnarray*}
    |\sigma_{\infty}(q;w)|&=&\frac{\Delta^{1/2}}{D^{1/2}}
    \left|\int_{\R^2}w_+(M\x^2)dx_1dx_2\right|\\
&\le& 2\frac{\Delta^{1/2}}{D^{1/2}}\int_{|M\x^2|\le 1}dx_1dx_2\\
    &\le& 2e^{4/3}\frac{\Delta^{1/2}}{D^{1/2}}\int_{|M\x^2|\le 1}
    \exp\{-1/(1-||M\x||^2/4)\}dx_1dx_2\\
    &\le& 2e^{4/3}
    \frac{\Delta^{1/2}}{D^{1/2}}\int_{\R^2}w_0(\tfrac12 M\x^2)dx_1dx_2\\
 &=& 4e^{4/3}
    \frac{\Delta^{1/2}}{D^{1/2}}\int_{\R^2}w_0(M\x^2)dx_1dx_2.
  \end{eqnarray*}
  On the other hand, applying Lemma \ref{MT} to the weight $w_0$ we
  find that
  \[\sigma_{\infty}(q;w_0)=2\frac{\Delta^{1/2}}{D^{1/2}}
  \int_{\R^2}w_0(M\x^2)dx_1dx_2.\]
  The first claim of the lemma then follows.

  For the remainder of the proof it will be convenient to write
  \[I(T)=\int_{\R^3}w_0(\y)K_T(q(\y))dy_1dy_2dy_3.\]
Now let $U$ be a real orthogonal matrix diagonalising $q(\x)$, so that
$q(U\x)=\mathrm{Diag}(\lambda,-\mu,-\nu)$ say.
Substituting $U\y$ in place of $\y$, and noting that the weight $w_0$
is invariant under rotations, we deduce that
\[I(T)= \int_{\R^3}w_0(\y)K_T(\lambda y_1^2-\mu y_2^3-\nu y_3^2)dy_1dy_2dy_3,\]
whence
\[I(T)\le \int_{[-1,1]^3}K_T(\lambda y_1^2-\mu y_2^3-\nu y_3^2)dy_1dy_2dy_3\]
and
\[I(T)\gg \int_{[-1/2,1/2]^3}K_T(\lambda y_1^2-\mu y_2^3-\nu y_3^2)dy_1dy_2dy_3\]
We now consider three cases. Firstly, suppose that $\lambda\ge \mu$. Then if
$\y\in[-1,1]^3$ and $K_T(\lambda y_1^2-\mu y_2^3-\nu y_3^2)\not=0$ we have
\[\lambda y_1^2\le T^{-1}+\mu y_2^2+\nu y_3^2\le T^{-1}+2\mu.\]
It follows that $|y_1|\le 2\sqrt{\mu/\lambda}$ as soon as
$T\ge (2\mu)^{-1}$.  Writing $\xi=\min(1\,,\,2\sqrt{\mu/\lambda}$ we
then deduce that
\begin{eqnarray*}
  I(T)&\le&
  \int_{[-1,1]^2}\int_{-\xi}^{\xi}
  K_T(\lambda y_1^2-\mu y_2^3-\nu y_3^2)dy_1dy_2dy_3\\
  &=&\sqrt{\mu/\lambda}  \int_{[-1,1]^2}
  \int_{-\xi\sqrt{\lambda/\mu}}^{\xi\sqrt{\lambda/\mu}}
  K_T(\mu y^2-\mu y_2^3-\nu y_3^2)dydy_2dy_3,
\end{eqnarray*}
on substituting $y_1=\sqrt{\mu/\lambda}y$. Since
$\xi\sqrt{\lambda/\mu}\le 2$ we obtain
\begin{eqnarray*}
  I(T)&\le&\sqrt{\mu/\lambda}  \int_{[-2,2]^3}
K_T(\mu y^2-\mu y_2^3-\nu y_3^2)dydy_2dy_3\\
&=&8\sqrt{\mu/\lambda}  \int_{[-1,1]^3}
K_T(4\mu z_1^2-4\mu z_2^3-4\nu z_3^2)dz_1dz_2dz_3\\
&=&8\sqrt{\mu/\lambda}.(4\mu)^{-1}\int_{[-1,1]^3}
K_{4\mu T}(z_1^2-z_2^3-\nu\mu^{-1}z_3^2)dz_1dz_2dz_3,
\end{eqnarray*}
for $T$ sufficiently large.  Similarly when
$\lambda\ge \mu$ we have
\begin{eqnarray*}
  I(T)&\gg &
  \int_{[-1/2,1/2]^2}\int_{-\xi}^{\xi}
  K_T(\lambda y_1^2-\mu y_2^3-\nu y_3^2)dy_1dy_2dy_3\\
  &=&\sqrt{\mu/\lambda}  \int_{[-1/2,1/2]^2}
  \int_{-\xi\sqrt{\lambda/\mu}}^{\xi\sqrt{\lambda/\mu}}
    K_T(\mu y^2-\mu y_2^3-\nu y_3^2)dydy_2dy_3.
\end{eqnarray*}
Since $\xi\sqrt{\lambda/\mu}\ge 1$ we obtain
\begin{eqnarray*}
  I(T)&\gg &  \sqrt{\mu/\lambda} \int_{[-1/2,1/2]^3}
  K_T(\mu y^2-\mu y_2^3-\nu y_3^2)dydy_2dy_3\\
  &=&\tfrac18 \sqrt{\mu/\lambda}  \int_{[-1,1]^3}
K_T(\tfrac14\mu z_1^2-\tfrac14\mu z_2^3-\tfrac14\nu z_3^2)dz_1dz_2dz_3\\
&=&8\sqrt{\mu/\lambda}.(\tfrac14\mu)^{-1}\int_{[-1,1]^3}
K_{\tfrac14\mu T}(z_1^2-z_2^3-\nu\mu^{-1}z_3^2)dz_1dz_2dz_3.
\end{eqnarray*}
Hence if we write
\[J(T;\delta)=\int_{[-1,1]^3}
K_T(z_1^2-z_2^2-\delta z_3^2)dz_1dz_2dz_3\]
we have
\[(\lambda\mu)^{-1/2}J(\tfrac14\mu T;\nu\mu^{-1})\ll I(T)\ll
(\lambda\mu)^{-1/2}J(4\mu T;\nu\mu^{-1}).\]
On taking the $\limsup$ as $T\to\infty$ this yields
\beql{ls1}
\lim_{T\to\infty}I(T)\asymp
(\lambda\mu)^{-1/2}\limsup_{T\to\infty}J(T;\nu\mu^{-1})
\eeq
when $\lambda\ge\mu$.

Suppose next that $\mu\ge\lambda\ge\nu$. Then if
$K_T(\lambda y_1^2-\mu y_2^3-\nu y_3^2)\not=0$  with $\y\in[-1,1]^3$ we have
\[\mu y_2^2\le \mu y_2^2+\nu y_3^2\le T^{-1}+\lambda y_1^2\le T^{-1}+\lambda.\]
It follows that $|y_2|\le 2\sqrt{\lambda/\mu}$ as soon as
$T\ge \lambda^{-1}$.  We may then replace the range $[-1,1]$ for $y_2$
by $[-\xi,\xi]$ where $\xi=\min(1\,,\,2\sqrt{\lambda/\mu})$ this
time. Proceeding much as before we find that
\begin{eqnarray*}
  I(T)&\le&\sqrt{\lambda/\mu}  \int_{[-2,2]^3}
K_T(\lambda y_1^2-\lambda y^3-\nu y_3^2)dy_1dydy_3\\
&=&8\sqrt{\lambda/\mu}.(4\lambda)^{-1}\int_{[-1,1]^3}
K_{4\lambda T}(z_1^2-z_2^3-\nu\lambda^{-1}z_3^2)dz_1dz_2dz_3
\end{eqnarray*}
and
\[I(T)\gg(\lambda\mu)^{-1/2}\int_{[-1,1]^3}
K_{4\lambda T}(z_1^2-z_2^3-\nu\lambda^{-1}z_3^2)dz_1dz_2dz_3.\]
Thus
\beql{ls2}
\lim_{T\to\infty}I(T)\asymp
(\lambda\mu)^{-1/2}\limsup_{T\to\infty}J(T;\nu\lambda^{-1})
\eeq
when $\mu\ge\lambda\ge\nu$.

Thirdly we suppose that $\lambda\le\nu$. In this case if 
$\y\in[-1,1]^3$ and $K_T(\lambda y_1^2-\mu y_2^3-\nu y_3^2)\not=0$ we have 
%both 
$|y_2|\le 2\sqrt{\lambda/\mu}$ and
$|y_3|\le 2\sqrt{\lambda/\nu}$ when
$T\ge \lambda^{-1}$.  We then replace the ranges for $y_2$ and $y_3$ by
by $[-\xi_2,\xi_2]$ and $[-\xi_3,\xi_3]$ respectively, where
$\xi_2=\min(1\,,\,2\sqrt{\lambda/\mu})$ and
$\xi_3=\min(1\,,\,2\sqrt{\lambda/\nu})$. A similar argument to before
then shows that
\beql{ls3}
\lim_{T\to\infty}I(T)\asymp
(\mu\nu)^{-1/2}\limsup_{T\to\infty}J(T;1)
\eeq
when $\nu\ge\lambda$.

It remains to consider $J(T;\delta)$, where $0<\delta\le 1$. To obtain
a lower bound we restrict the variables $z_1,z_2$ to the square given by
$|z_1+z_2|\le 1$ and $|z_1-z_2|\le 1$, which lies inside the region
$[-1,1]^2$. A suitable change of variable then shows that
\[J(T;\delta)\ge \frac{1}{2}\int_{[-1,1]^3}
K_T(u_1u_2-\delta z_3^2)du_1du_2dz_3.\]
We now restrict $u_1,u_2$ further, so that $u_1,u_2\ge 0$ and
$\delta/4\le u_1u_2\le\delta/2$. For any such $u_1,u_2$ and any
$T\ge 4\delta^{-1}$ the inequality 
\[|u_1u_2-\delta z^2|\le 1/(2T)\]
implies
\[\delta z^2\le u_1u_2+(2T)^{-1}\le\delta,\]
whence one automatically has $z\in[-1,1]$. Moreover it also implies
\[\delta z^2\ge u_1u_2-(2T)^{-1}\ge\delta/8,\]
whence $|z|\ge \tfrac13$, say.  It follows that one has
\[|u_1u_2-\delta z_3^2|\le 1/(2T)\]
for an admissible set of values for $z_3$ of measure $\gg (\delta
T)^{-1}$. We therefore conclude that
\begin{eqnarray}\label{JTlb}
  J(T;\delta)&\gg &\delta^{-1}\mathrm{Meas}\{(u_1,u_2)\in[0,1]^2:\,
  \delta/4\le u_1u_2\le\delta/2\}\nonumber\\
  &\gg &\delta^{-1}\int_{2/\delta}^1\frac{\delta}{4u_1}du_1\nonumber\\
  &\gg &\log(2/\delta).
\end{eqnarray}

To obtain an upper bound for $J(T;\delta)$ we extend the range of
$z_1,z_2$ to the square given by
$|z_1+z_2|\le 2$ and $|z_1-z_2|\le 2$. A suitable change of variable
now shows that
\[J(T;\delta)\le \frac{1}{2}\int_{[-2,2]^3}
K_T(u_1u_2-\delta z_3^2)du_1du_2dz_3.\]
When $|u_1u_2|\le 2/T$ the integrand is only non-zero when
$z_3^2\le 3/(\delta T)$, so that this range contributes
\[\ll T\int_{-2}^2\min\left(1\,,\,\frac{1}{|u_2|T}\right)(\delta T)^{-1}du_2
\ll\delta^{-1/2}\frac{\log T}{\sqrt{T}}.\]
The range $|u_1u_2|\le 2/T$ therefore makes no contribution when we
let $T$ go to infinity. When $|u_1u_2|\ge 2/T$ the integrand is only
non-zero when $\sqrt{|u_1u_2|/(2\delta)}\le|z_3|\le 2$.  Thus the set of values
for $z_3$ for which $|u_1u_2-\delta z_3^2|\le T^{-1}$ consists of at
most two intervals, having total length
$O(T^{-1}(|u_1u_2|\delta)^{-1/2})$. Moreover the set is empty unless
$|u_1u_2|\le 8\delta$.It follows that the corresponding contribution
to $J(T;\delta)$ is
\begin{eqnarray*}
  &\ll& T\int_{-2}^2\int_{|u_1|\le\min(2\,,\,8\delta/|u_2|)}
  T^{-1}(|u_1u_2|\delta)^{-1/2}du_1du_2\\
  &\ll&\delta^{-1/2}\int_{-2}^2
  \frac{\min(2\,,\,8\delta/|u_2|)^{1/2}}{|u_2|^{1/2}}du_2\\
  &\ll&\log(2/\delta).
\end{eqnarray*}
In view of the lower bound (\ref{JTlb}) we therefore have
$J(T;\delta)\asymp \log(2/\delta)$, and the second claim of Lemma
\ref{cqE} then follows from (\ref{ls1}), (\ref{ls2}) and (\ref{ls3}).

For the third claim of the lemma we note that
\[\min(\lambda,\mu,\nu)\ge\frac{\lambda\mu\nu)}{\max(\lambda,\mu,\nu)}
\gg\frac{\Delta}{||q||^2}.\]
This produces the lower bound for $\sigma_{\infty}(q,w_0)$.  For the
upper bound we observe for example that when
$\lambda\ge\mu \ge\nu$ we have
\[\min(\lambda,\mu,\nu)^{1/2}\log(2\mu/\nu)=\nu^{1/2}\log(2\mu/\nu)
\ll\nu^{1/2}(\mu\nu)^{1/4}\]
and that
\[(\mu\nu)^{1/4}\ll\Delta^{1/4}\lambda^{-1/4}\ll\Delta^{1/4}||q||^{-1/4}.\]
Thus when $\lambda\ge\mu \ge\nu$ we have
\[\sigma_{\infty}(q;w_0)\ll\Delta^{-1/4}||q||^{-1/4}=\rho^{1/4}||q||^{-1}.\]
When $\mu\ge\lambda\nu$ or $\mu\ge\nu\ge\lambda$ we may argue similarly.
This suffices to complete the proof of the lemma.
\end{proof}

\section{Theorem \ref{t4} --- The error term}

In Lemma \ref{poisson}
the contribution from $\bfa\not=\mathbf{0}$ will produce an error
term, as we now show.
\begin{lemma}\label{l:err}
  Let $\Lambda\subseteq\Z^2$ be a 2-dimensional lattice, and let
  $\b{n}_1,\b{n}_2$ be a basis for $\Lambda$ chosen so that
  $||\b{n}_1||\cdot||\b{n}_2||\ll\dL(\Lambda)$. Then if
  $N=(\b{n}_1|\b{n}_2)$ we have $\Lambda=N(\Z^2)$. Moreover for any
  integer $K\ge 0$ we have
  \[I(\bfa,M_k,N)\ll_{w,K}BD_k^{-1/2}||q||^{1/2}
  \left(\frac{B_0}{B}\right)^{K/2}||\bfa||^{-K},\]
  with
  \[B_0=\rho\,\dL(\Lambda)^2||\bfc_3||.\]
  \end{lemma}

\begin{proof}
  Let $N^{-1}\bfa=(b_1,b_2)^T$ and suppose that $|b_1|\ge|b_2|$, say.  If we
  integrate by parts $K$ times with respect to $x_1$ we find that
\[I(\bfa,M_k,N)\ll_K|b_1|^{-K}\int_{\R^2}\left|\frac{\partial^K}{\partial x_1^K}
w(B^{-1}M_k\x^2)\right|dx_1dx_2.\]
Since $\bfa=N(b_1,b_2)^T$ we have
\[||\bfa||\ll\max(||\b{n}_1||,||\b{n}_2||)\max(|b_1|,|b_2|)\ll
\dL(\Lambda)|b_1|,\]
whence
\[I(\bfa,M_k,N)\ll_K ||\bfa||^{-K}\dL(\Lambda)^K
\int_{\R^2}\left|\frac{\partial^K}{\partial x_1^K}
w(B^{-1}M_k\x^2)\right|dx_1dx_2.\]
  We will write the components of $B^{-1}M_k\x^2$ as
  $f_1(x_1),f_2(x_1),f_3(x_1)$, where the $f_i$ are quadratic
  polynomials which also involve $x_2$.  Then the $K$-th order partial
  derivatives of $w(B^{-1}M_k\x^2)$ with respect to $x_1$ will be sums
  of various terms $T_n$.  Each $T_n$ will be a product containing a
  single partial derivative of $w$, of order at most $K$, along with
  various first and second derivatives of the $f_i$.  If there are $r$
  first derivatives and $s$ second derivatives then $r+2s=K$. It
  therefore follows that
  \[\frac{\partial^K}{\partial x_1^K}w(B^{-1}M_k\x^2)\ll_{w,K} F_1^rF_2^s\]
  for some exponents with $r+2s=K$, where
  \[F_1=\sup\{|f_i'(x_1)|:\, w(B^{-1}M_k\x^2)\not=0\},\]
and $F_2$ is the maximum of $|f_1''|,|f_2''|$ and $|f_3''|$.

The leading coefficient of $Bf_i$ will be the $i$-th entry in the first
column of $M_k$, so that its modulus will be at most $||\b{c}_1||$, in
the notation of Theorem \ref{t2}.  It follows that
$F_2\le 2||\b{c}_1||/B$.  Similarly the coefficient of $x_1$ in $Bf_i$ will have
modulus at most $||\b{c}_2||$, so that
\[F_1\le X_1||\b{c}_1||/B+X_2||\b{c}_2||/B.\]
It follows via Lemma \ref{IE} that
\begin{eqnarray*}
  \max_{r+2s=K}F_1^rF_2^s&\le& F_1^K+F_2^{K/2}\\
  &\ll_K& (X_1||\b{c}_1||/B)^K+(X_2||\b{c}_2||/B)^K+(||\b{c}_1||/B)^{K/2}\\
&\ll_K &B^{-K/2}\{||\b{r}_1||^{K/2}||\b{c}_1||^K+||\b{r}_3||^{K/2}||\b{c}_2||^K
  +||\b{c}_1||^{K/2}\}\\
  &\ll_K &  B^{-K/2}E^{K/2},
\end{eqnarray*}
with
\[E=||\b{r}_1||\cdot||\b{c}_1||^2+||\b{r}_3||\cdot||\b{c}_2||^2+||\b{c}_1||.\]
Lemma \ref{IE} shows that the support of $w(B^{-1}M\x^2)$ is included
in a rectangle of area
$O(BD_k^{-1/2}||q||^{1/2})$, and we therefore conclude that
\beql{Ie1}
I(\bfa,M_k,N)\ll_{w,K}\frac{B||q||^{1/2}}{D_k^{1/2}}
\left\{\dL(\Lambda)^2B^{-1}E\right\}^{K/2}||\bfa||^{-K}.
\eeq

We now claim that
\beql{cE}
E\ll \rho\,||\b{c}_1||\ll \rho\,||\bfc_3||
\eeq
when $|b_1|\ge|b_2|$, as we are currently supposing.  In the
alternative case the argument is completely analogous, leading to
exactly the same bound.  To establish our claim we
use (\ref{c1E}), (\ref{t2crE}) and (\ref{MD}) to show that
\begin{eqnarray*}
  ||\b{r}_1||\cdot||\b{c}_1||^2&\le&
\det(M_k)||\b{r}_1||\cdot||\b{r}_2||\cdot||\b{r_3}||\cdot||\b{c}_1||\\
  &\le&9\sqrt{10}\det(M_k)||q||^{3/2}D_k^{-3/2}||\b{c}_1||\\
  &\ll& \rho^{1/2}||\b{c}_1||.
\end{eqnarray*}
This is sufficient for the term $||\b{r}_1||\cdot||\b{c}_1||^2$, since
$\rho\ge 2$ by Theorem \ref{t5}.  Secondly,
the bounds (\ref{alp}), (\ref{t2crE}),
(\ref{t22}) and (\ref{MD}) yield
\begin{eqnarray*}
  ||\b{r}_3||\cdot||\b{c}_2||^2
  &\le&
  ||\b{r}_1||\cdot||\b{c}_1||\cdot||\b{r}_3||\cdot||\b{c}_2||^2\\
&\ll & D_k^{-1}||q||\cdot||\b{c}_1||\cdot||\b{c}_2||^2\\
&\ll & \det(M_k)^2D_k^{-3}||q||^3||\b{c}_1||\\
 &\ll& \rho||\b{c}_1||,
\end{eqnarray*}
  which is sufficient for the term $||\b{r}_3||\cdot||\b{c}_2||^2$.
  Finally,  since $\rho\ge 2$ we have
  $||\b{c}_1||\le\rho||\b{c}_1||$.  This give us
  the required estimate (\ref{cE})
for $E$, whence (\ref{Ie1}) produces the bound
\[I(\bfa,M_k,N)\ll_{w,K}||\bfa||^{-K}\frac{B||q||^{1/2}}{D_k^{1/2}}
\left\{\frac{\dL(\Lambda)^2}{B}\rho\,||\bfc_3||\right\}^{K/2}.\]
The lemma now follows.
\end{proof}

We can now summarize the results of our analysis of
$S(\Lambda,B,M_k)$. 
\begin{lemma}\label{prop}
  We have $S(\Lambda,B,M_k)=0$ if $B<1$, and otherwise
  \[  S(\Lambda,B,M_k)\ll_w D_k^{-1/2}||q||^{1/2}\left\{\frac{B}{\dL(\Lambda)}+
  B^{1/2}||\b{c}_3||^{1/2}\right\}.\]

  Moreover
  \begin{eqnarray}\label{as}
  S(\Lambda,B,M_k)&=&
  \sigma_{\infty}(q;w)\frac{\Delta^{1/2}}{D_k^{1/2}\dL(\Lambda)}B+
O_w(1)\nonumber\\
&&\hspace{1cm}{}+O_{w,K}(D_k^{-1/2}||q||^{1/2}B(B_0/B)^K),
  \end{eqnarray}
for any integer $K\ge 2$,  with
  \[B_0=\rho\,\dL(\Lambda)^2||\bfc_3||.\]
\end{lemma}
\begin{proof}
  The first half is the content of Lemma \ref{CB}, while the second
  follows from Lemma \ref{poisson}, together with Lemma \ref{MT} for
  the term $\bfa=\mathbf{0}$, and Lemma \ref{l:err} for
  $\bfa\not=\mathbf{0}$.  Here we replace $K$ by $2K$ and observe that
  \[\twosum{\bfa\in\Z^2}{\bfa\not=0}||\bfa||^{-2K}\ll_K 1\]
  for $K\ge 2$.
\end{proof}

\section{Completing the proof of Theorem \ref{t4}}

To prove Theorem \ref{t4} we will apply Lemma \ref{lSlat}, using the 
crude upper bound from Lemma \ref{prop} when $\dL(\Lambda^{(j)})$ or
$d$  is large, and the asymptotic estimate (\ref{as}) otherwise. We therefore
begin by choosing a real parameter $d_0\ge 1$, which we will specify
later, and noting that we can restrict attention to the range
$d\le\sqrt{B}$, by virtue of the first clause of Lemma~\ref{prop}.

The contribution to $N(B;\cl{C}_k)$ from terms with $d_0\le d\le\sqrt{B}$,
summed over all the lattices $\Lambda^{(j)}$, will be
\begin{eqnarray*}
&\ll_w& 3^{\omega(\Delta)}\frac{||q||^{1/2}}{D_k^{1/2}}
\sum_{d_0\le d\le\sqrt{B}}\left\{Bd^{-2}+
B^{1/2}d^{-1}||\b{c}_3||^{1/2}\right\}\\
&\ll_w&
3^{\omega(\Delta)}\frac{||q||^{1/2}}{D_k^{1/2}}\left\{Bd_0^{-1}+
B^{1/2}(\log B)||\b{c}_3||^{1/2}\right\}.
\end{eqnarray*}
Similarly, the contribution from terms with $\dL(\Lambda^{(j)})\ge d_0$ 
and $d\le d_0$ will be
\begin{eqnarray*}
&\ll_w&\frac{||q||^{1/2}}{D_k^{1/2}}
  \sum_{\substack{j=1\\ \dL(\Lambda^{(j)})\ge  d_0}}^J
\sum_{d\le d_0}\left\{\frac{Bd^{-2}}{\dL(\Lambda^{(j)})}+
B^{1/2}d^{-1}||\b{c}_3||^{1/2}\right\}\\
&\ll_w&
3^{\omega(\Delta)}\frac{||q||^{1/2}}{D_k^{1/2}}\left\{Bd_0^{-1}+
B^{1/2}(\log B)||\b{c}_3||^{1/2}\right\}.
\end{eqnarray*}

We now examine the terms for which both $d\le d_0$ and
$\dL(\Lambda^{(j)})\le d_0$.  We have
\[\sum_{\substack{j=1\\ \dL(\Lambda^{(j)})\ge d_0}}^J\frac{1}{\dL(\Lambda^{(j)})}\,
\sum_{d=1}^{\infty}\frac{1}{d^2}\ll 3^{\omega(\Delta)}d_0^{-1},\]
and
\[\sum_{j=1}^J\frac{1}{\dL(\Lambda^{(j)})}\,\sum_{d\ge d_0} \frac{1}{d^2}
\ll 3^{\omega(\Delta)}d_0^{-1}.\]
Hence, when we use the asymptotic formula (\ref{as}) for terms in which
both $d$ and $\dL(\Lambda^{(j)})$ are at most $d_0$, the main term contributes
\[\tfrac12\sigma_{\infty}(q;w)\frac{\Delta^{1/2}}{D_k^{1/2}}B
\left\{\kappa+O(3^{\omega(\Delta)}d_0^{-1})\right\},\]
where
\[\kappa=\sum_{j=1}^J\lambda\left(\dL(\Lambda^{(j)})\right)
\sum_{\substack{d=1\\ (d,\Delta_1\Delta_2)=1}}^{\infty}\mu(d)d^{-2}=
\frac{6}{\pi^2}\prod_{p\mid \Delta_1}\frac{1}{1+p^{-1}}
\prod_{p\mid \Delta_2}\frac{1-p^{-1}}{1+p^{-1}},\]
by Lemma \ref{lSlat}.
Using the estimate from Lemma \ref{cqE} we see that the $O$-term above
contributes
\[\ll_w 3^{\omega(\Delta)}\frac{\rho^{1/4}}{||q||}
\frac{\Delta^{1/2}}{D_k^{1/2}}\frac{B}{d_0}=
3^{\omega(\Delta)}\rho^{-1/4}\frac{||q||^{1/2}}{D_k^{1/2}}\frac{B}{d_0}
\le 3^{\omega(\Delta)}\frac{||q||^{1/2}}{D_k^{1/2}}\frac{B}{d_0}.\]

On the other hand, the error term $O_w(1)$ in (\ref{as}) contributes
$O_w(d_0^2)$ while the second error term contributes
\begin{eqnarray*}
&\ll_{w,K}&\frac{||q||^{1/2}}{D_k^{1/2}}B(B_1/B)^K
\sum_{\substack{j=1\\ \dL(\Lambda^{(j)})\le d_0}}^J
\sum_{d\le d_0}\dL(\Lambda^{(j)})^{2K}d^{2K-2}\\
&\ll_{w,K}&\frac{||q||^{1/2}}{D_k^{1/2}}B(B_1/B)^K3^{\omega(\Delta)}d_0^{4K-1},
\end{eqnarray*}
with
\[B_1=\rho\,||\bfc_3||.\]
Thus, if we assume that $B\ge B_1$ we may take $d_0=B^{1/4}B_1^{-1/4}$
and $K=2$ so that the total of all the above error terms is
\[\ll_w d_0^2+3^{\omega(\Delta)}\frac{||q||^{1/2}}{D_k^{1/2}}
  \left\{\frac{B}{d_0}+B^{1/2}(\log B)||\b{c}_3||^{1/2}\right\}.\]
  In the notation of Theorem \ref{t5} we have
\beql{bc3b}
||\b{c}_3||^2\ge||\z_1||\cdot||\z_2||\ge D_k/||q||,
\eeq
whence
\[d_0^2=\frac{B^{1/2}}{B_1^{1/2}}\ll \frac{B^{1/2}}{||\b{c}_3||^{1/2}}
\ll \frac{||q||^{1/2}}{D_k^{1/2}}B^{1/2}||\b{c}_3||^{1/2}.\]
Thus the error term $d_0^2$ above is dominated by the final term.
Moreover 
\[\frac{B}{d_0}=B^{3/4}B_1^{1/4}\gg
  B^{1/2}\max\{||\b{c}_1||\,,\,||\b{c}_3||\}^{1/2}\]
  when $B\ge B_1$.  We therefore deduce that
  \[  N(B;\cl{C}_k)=\tfrac12
  \sigma_{\infty}(q,w)\kappa\frac{\Delta^{1/2}}{D_k^{1/2}}B+
  O_w\left(3^{\omega(\Delta)}\frac{||q||^{1/2}}{D_k^{1/2}}
  B_1^{1/4}B^{3/4}\log B\right)\]
  for $B\ge B_1$.

  When $\rho^{-1}||\b{c}_3||\ll B\le B_1$ we argue as above with
  $d_0=1$, showing that
  \[  N(B;\cl{C}_k)\ll_w 1+3^{\omega(\Delta)}\frac{||q||^{1/2}}{D_k^{1/2}}
  \left\{B+B^{1/2}(\log B)||\b{c}_3||^{1/2}\right\}.\]
  However for $B\le B_1$ we have
  \[\sigma_{\infty}(q,w)\kappa\frac{\Delta^{1/2}}{D_k^{1/2}} B\ll
  3^{\omega(\Delta)}\frac{||q||^{1/2}}{D_k^{1/2}}B_1^{1/4}B^{3/4}\]
by Lemma \ref{cqE}, and $B\ll B_1^{1/4}B^{3/4}$. Moreover
$B^{1/2}||\b{c}_3||^{1/2}\ll B_1^{1/4}B^{3/4}$ when $B\gg \rho^{-1}||\b{c}_3||$.
It follows that
\begin{eqnarray}\label{NBE}
N(B;\cl{C}_k)&=&\tfrac12
  \sigma_{\infty}(q,w)\kappa\frac{\Delta^{1/2}}{D_k^{1/2}}B+O_w(1)\nonumber\\
  &&\hspace{5mm}{}+
  O_w\left(3^{\omega(\Delta)}\frac{||q||^{1/2}}{D_k^{1/2}}
  B_1^{1/4}B^{3/4}\log B\right)
  \end{eqnarray}
in the range $\rho^{-1}||\b{c}_3||\ll B\le B_1$.
On the other hand, if we have $B< (90\rho)^{-1}||\b{c}_3||$ then $||\x||\le B$
implies $||\x||<||\b{z}_2||$, in the notation of Theorem \ref{t5}, so
that $N(B;\cl{C}_k)$ counts at most the points $\pm\b{z}_1$.  In this
case we will have $N(B;\cl{C}_k)\ll_w 1$, and hence (\ref{NBE}) holds
for $B< (90\rho)^{-1}||\b{c}_3||$ too. To complete the proof of
Theorem \ref{t4} it remains to observe that $\kappa\gg
(3/4)^{\omega(\Delta)}$, that $\sigma_{\infty}(q,w)\gg ||q||^{-1}$, by
Lemma \ref{cqE}, and that $||\b{c}_3||\ll\rho\,||\b{z}_2||$, by
Theorem\ref{t5}.

\bigskip

\bigskip

\end{document}